\documentclass[11pt]{article}
\usepackage{amsmath,amsthm,amsfonts,amssymb,indentfirst,xspace,bm,graphicx}
\usepackage{framed,paralist,subfigure,IEEEtrantools}
\usepackage{mathtools}
\usepackage[pdfborder={0 0 .1}]{hyperref}
\usepackage[initials,nobysame]{amsrefs}

\usepackage[body={16.65cm,22.65cm}]{geometry}

\makeatletter
\renewenvironment{proof}[1][\proofname]{\par
  \pushQED{\qed}%
  \normalfont \topsep6\p@\@plus6\p@\relax
  \trivlist
  \item[\hskip\labelsep
        \bfseries
    #1\@addpunct{.}]\ignorespaces
}{%
  \popQED\endtrivlist\@endpefalse
}
\makeatother

    {\endMakeFramed}
\newcommand{\del}{\partial}

\newcommand{\lap}{\Delta}

\newcommand{\inv}{^{-1}}

\newcommand{\grad}{\nabla}

\newcommand{\divergence}{\grad \cdot}

\newcommand{\tauann}{\tau_{\text{ann}}}
\newcommand{\etal}{\eta_{\text{sub}}}
\newcommand{\bartauann}{\bar\tau_{\text{ann}}}
\newcommand{\annulus}{\mathcal A}
\newcommand{\qavg}{q^{\text{avg}}}

\newcommand{\eps}{\varepsilon}

\renewcommand{\epsilon}{\varepsilon}
\renewcommand{\leq}{\leqslant}
\renewcommand{\geq}{\geqslant}

\DeclareMathOperator{\sign}{sign}
\DeclareMathOperator{\trace}{tr}

\newcommand{\R}{\mathbb{R}}

\newif\iftextstyle
\textstyletrue
\everydisplay\expandafter{\the\everydisplay\textstylefalse}

\DeclarePairedDelimiter{\abs}{\lvert}{\rvert}
\DeclarePairedDelimiter{\norm}{\lVert}{\rVert}
\DeclarePairedDelimiter{\average}{\langle}{\rangle}

\newcommand{\st}{\;\iftextstyle|\else\big|\fi\;}
\newcommand{\defeq}{\stackrel{\text{def}}{=}}

\numberwithin{equation}{section}
\allowdisplaybreaks

\newtheorem{theorem}{Theorem}[section]

\newtheorem{lemma}[theorem]{Lemma}
\newtheorem{proposition}[theorem]{Proposition}

\newtheorem*{theorem*}{Theorem}
\newtheorem*{lemma*}{Lemma}
\newtheorem*{proposition*}{Proposition}
\newtheorem*{corollary*}{Corollary}

\theoremstyle{definition}

\theoremstyle{remark}
\newtheorem{remark}[theorem]{Remark}
\newtheorem*{remark*}{Remark}

\newcommand{\commentout}[1]{}

\begin{document}
\title{From homogenization to averaging  in cellular flows}
\author{Gautam Iyer\thanks{Department of 
Mathematical Sciences, Carnegie Mellon University, Pittsburgh, PA 15213; gautam@math.cmu.edu}
\and Tomasz Komorowski\thanks{Institute of Mathematics,  UMCS,
pl. Marii Curie-Sk\l odowskiej 1,
20-031, Lublin and
IMPAN,
ul. \'{S}niadeckich 8,   00-956 Warsaw, Poland, e-mail:
komorow@hektor.umcs.lublin.pl}
\and Alexei Novikov\thanks{Department of Mathematics, Pennsylvania State University, State College PA 16802,
anovikov@math.psu.edu}\and
Lenya Ryzhik\thanks{Department of Mathematics, Stanford University,
Stanford, CA 94305, USA,
e-mail: ryzhik@math.stanford.edu}}

\maketitle

\begin{abstract}
We consider an elliptic eigenvalue problem with a fast cellular flow of amplitude $A$, in a two-dimensional domain with 
$L^2$ cells. For fixed $A$, and $L \to \infty$, the problem homogenizes, and has been well studied. Also well studied is the limit when $L$ is fixed, and $A \to \infty$. In this case the solution equilibrates along stream lines.

In this paper, we show that if \textit{both} $A \to \infty$ and $L \to \infty$, then a transition between the homogenization and averaging regimes 
occurs at $A \approx L^4$. When $A\gg L^4$, the principal Dirichlet eigenvalue is approximately constant. On the other hand, when $A\ll L^4$, the principal eigenvalue behaves like ${\bar \sigma(A)}/L^2$, where $\bar \sigma(A) \approx \sqrt{A} I$ is the effective diffusion matrix. A similar transition is observed for the solution of the exit time problem. The proof in the homogenization regime involves bounds on 
the second correctors. Miraculously, if the slow profile is quadratic, these estimates can be obtained using  
drift independent $L^p \to L^\infty$ estimates for elliptic equations with an incompressible drift. This provides effective sub and super-solutions for our problem.
\end{abstract}
\section{Introduction}

Consider an advection diffusion equation of the form
\begin{equation}\label{intro2}
\del_t \varphi + Av(x)\cdot\nabla\varphi-\Delta\varphi = 0.
\end{equation}
where $A$ is the non-dimensional strength of a prescribed vector field $v(x)$. Under reasonable assumptions when $A \to \infty$, the solution~$\varphi$ becomes constant on the trajectories of $v$. Indeed, dividing \eqref{intro2} by $A$ and passing to the limit $A\to\infty$ formally shows
\[
v(x)\cdot\nabla\varphi=0,
\]
which, of course, forces $\varphi$ to be constant along trajectories of $v$. Well known ``averaging'' results~\cites{FW,Kifer,PS} study the slow evolution of $\varphi(t,x)$ across various trajectories.

On the other hand, if we fix $A = 1$, classical homogenization results~\cites{BLP,KOZ,PS} determine the long time behavior of solutions of \eqref{intro2}. For such results it is usually convenient to choose $\eps\ll 1$ small, and rescale~\eqref{intro2} to time scales of order $1/\epsilon^2$, and distance scales of order $1/\epsilon$. This gives
\begin{equation}\label{intro4}
\del_t \varphi_\eps +\frac{1}{\eps}v\left(\frac{x}{\eps}\right)\cdot\nabla\varphi_\eps  - \Delta\varphi_\eps = 0.
\end{equation}
Assuming $v$ is periodic, and that the initial condition varies slowly (i.e.\ $\varphi_\epsilon( x, 0)$ is independent of $\epsilon$), standard homogenization results show that $\varphi_\epsilon \to \bar\varphi$, as $\epsilon \to 0$. Further, $\bar\varphi$ is the solution of the effective problem
\begin{equation}\label{intro6}
\frac{\del\bar\varphi}{\del t}=\nabla\cdot(\bar \sigma\nabla\varphi),
\end{equation}
and $\bar\sigma$ is the effective diffusion matrix, which can be computed as follows. Define the correctors $\chi_1$, \dots, $\chi_n$ to be the mean-zero periodic solutions of
\begin{equation}\label{intro8}
-\Delta\chi_j+v(x)\cdot\nabla\chi_j=-v_j(x), \quad j=1,\dots,n.
\end{equation}
Then
\begin{equation}\label{intro10}
\bar \sigma_{ij}= \delta_{ij} + \frac{1}{\abs{Q}} \int_{Q}\nabla\chi_i\cdot\nabla\chi_j \,dx, \quad i,j=1,\dots,n.
\end{equation}
$Q$ is the period cell of the flow $v(x)$, and $\delta_{ij}$ is the Kronecker delta function.

The main focus of this paper is to study a transition between the two well known regimes described above. To this end, rescale~\eqref{intro2} by choosing time scales of the order $1/\epsilon^2$ and length scales of order $1/\epsilon$. This gives
\begin{equation}\label{intro12}
\del_t \varphi_{\eps,A} +\frac{A}{\eps}v\left(\frac{x}{\eps} \right)\cdot\nabla\varphi_{\eps,A} - \Delta\varphi_{\eps,A} = 0,
\end{equation}
where $A\gg 1$ and $\eps\ll 1$ are two \textit{independent} parameters. Of course, if we keep $\epsilon$ fixed, and send $A \to \infty$, the well known averaging results apply. Alternately, if we keep $A$ fixed and send $\epsilon \to 0$, we are in the regime of standard homogenization results. The present paper considers~\eqref{intro12} with \textit{both} $\epsilon \to 0$ and $A \to \infty$. Our main result shows that if $v$ is a $2D$ cellular flow, then we see a sharp transition between the homogenization and averaging regimes at $A \approx 1/\epsilon^4$.

Before stating our precise results (Theorems~\ref{thmEvalStrongFlow} and~\ref{thmEvalHomog} below), we provide a brief explanation as to why one expects the transition to occur at $A \approx 1/\epsilon^4$. For simplicity and concreteness, we choose the stream function $H(x_1,x_2) = \frac{1}{\pi} \sin(\pi x_1)\sin(\pi x_2)$, and define $v(x_1,x_2)=(-\del_2 H, \del_1 H )$. Even in this simple setting, to the best of our knowledge, the transition from averaging to homogenization has not been studied before.

First, for any \textit{fixed} $A$, we let $\bar\sigma(A) = (\bar \sigma_{ij}(A))$ denote the effective diffusion matrix obtained in the limit $\eps\to 0$ (see~\cite{KramerMajda} for a comprehensive review).
If $\chi_j^A$ is the mean zero, $2$-periodic solution to
\begin{equation}\label{intro16}
-\Delta\chi_j^A+Av(x)\cdot\nabla\chi_j= - Av_j(x), \quad\text{for } j \in \{1, 2\},
\end{equation}
then the effective diffusivity (as a function of $A$) is given by~\eqref{intro10}. As $A \to \infty$, the behaviour of the correctors $\chi_j^A$ is well understood~\cites{bblChildress,bblFannjiangPapanicolaou,Heinze,Koralov,bblNovikovPapanicolaouRyzhik,bblRosenbluth,bblShraiman,bblYoung}. Except on a boundary layer of order $1/\sqrt{A}$, each of the functions~$\chi_j(x) + x_j$ become constant in cell interiors. Using this one can show (see for instance~\cites{bblChildress,bblFannjiangPapanicolaou}) that asymptotically, as $A \to \infty$,  the effective diffusion matrix behaves like
\begin{equation}\label{eqnEffectiveDiffusivityAsym}
\bar \sigma(A) = \sigma_0\sqrt{A} I + o(\sqrt{A})
\end{equation}
Here $I$ is the identity matrix, and $\sigma_0 > 0$ is an explicitly computable constant. Consequently, if we consider~\eqref{intro12}, with 
the Dirichlet boundary conditions on the unit square, we expect
\begin{equation}\label{eqnIntroHomog}
\varphi_{\epsilon,A} (x, t) \approx \exp( - \sigma_0 \sqrt{A} t ), \quad\text{as  $t \to \infty$, for small $\epsilon$.}
\end{equation}

On the other hand, if we keep $\epsilon$ fixed and send $A \to \infty$, we know~\cites{FW,bblYoung} that~$\varphi$ becomes constant on stream lines of $H$. In particular, because of the
Dirichlet boundary condition on the outside boundary, we must also have $\varphi = 0$ on the   boundary of all interior cells. Since these cells have side length $\epsilon$, we expect
\begin{equation}\label{eqnIntroSF}
\varphi_{\epsilon,A} (x, t) \approx \exp( - \pi^2 t / \epsilon^2 ), \quad\text{as  $t \to \infty$ for large $A$.}
\end{equation}
Matching~\eqref{eqnIntroHomog} and~\eqref{eqnIntroSF} leads us to believe $\sqrt{A} \approx 1/\epsilon^2$ marks the transition between the two regimes.\medskip

With this explanation, we state our main results. Our first two results study the averaging to homogenization transition for the principal Dirichlet eigenvalue. Let $L$ be an even integer, $D=[-L/2,L/2]\times[-L/2,L/2]$ be a square of side length $L$, and $A > 0$ be given. We study the principal eigenvalue problem on $D$
\begin{equation}\label{eqnEval}
\left\{
\begin{aligned}
-\lap \varphi + A v \cdot \grad \varphi &= \lambda \varphi &&\text{in } D\\
\varphi &= 0 &&\text{on }\del D\\
\varphi &> 0 && \text{in }D,
\end{aligned}\right.
\end{equation}
as both $L, A \to \infty$. We observe two distinct behaviors of $\lambda$ with a sharp transition. If $A \gg L^4$, then the principal eigenvalue stays bounded, and can be read off using the variational principle in~\cite{bblBeresHamelNadirshvilli} in the limit  $A \to \infty$. This is the averaging regime, and exactly explains~\eqref{eqnIntroSF}. On the other hand, if $A \ll L^4$, then the principal eigenvalue is of the order ${\bar \sigma(A)}/{L^2}$. This is the homogenization regime, and when rescaled to a domain of size $1$, exactly explains~\eqref{eqnIntroHomog}. Our precise results are stated below.
\begin{theorem}[The averaging regime]\label{thmEvalStrongFlow}
Let $\varphi = \varphi_{L,A}$ be the solution of~\eqref{eqnEval} and $\lambda = \lambda_{L,A}$ be the principal eigenvalue. If $A \to \infty$, and $L = L(A)$ varies such that
\begin{equation}\label{eqnStrongFlowAssumption}
\liminf_{A \to \infty} \frac{\sqrt{A}}{L^2 \log A \log L} > 0,
\end{equation}
then there exist two constants $\lambda_0, \lambda_1$, independent of $L$ and $A$, such that
\begin{equation}\label{eqnStrongFlowEvalBounds}
0 < \lambda_0 \leq \lambda_{L,A} \leq \lambda_1 < \infty
\end{equation}
for all $A$ sufficiently large. 
\end{theorem}
\begin{theorem}[The homogenization regime]\label{thmEvalHomog}
As with Theorem~\ref{thmEvalStrongFlow}, let $\varphi = \varphi_{L,A}$ be the solution of~\eqref{eqnEval} and $\lambda = \lambda_{L,A}$ be the principal eigenvalue. If $L \to \infty$, and $A = A(L)$ varies such that
\begin{equation}\label{eqnHomogAssumption}
\frac{1}{c} L^{4 - \alpha} \leq A \leq c L^{4 - \alpha}, \quad\text{for some }\alpha > 0,
\end{equation}
then there exists a constant $C = C(\alpha, c) > 0$, independent of $L$ and  $A$, such that
\begin{equation}\label{eqnHomogEvalBounds}
\frac{1}{C} \frac{\sqrt{A}}{L^2} \leq \lambda_{L, A} \leq C \frac{\sqrt{A}}{L^2}
\end{equation}
for all $L$ sufficiently large.
\end{theorem}
\begin{remark*}
In the special case where $A = L^\beta$, for $\beta > 4$, assumption~\eqref{eqnStrongFlowAssumption} is satisfied, and consequently the principal eigenvalue remains bounded and non-zero. For $\beta < 4$, assumption~\eqref{eqnHomogAssumption} is satisfied and the principal eigenvalue behaves like that of the homogenized equation.
\end{remark*}

In the averaging regime (Theorem~\ref{thmEvalStrongFlow}), the proof of the upper bound in \eqref{eqnStrongFlowEvalBounds} follows directly using ideas of~\cite{bblBeresHamelNadirshvilli}. The lower bound, however, is much more intricate. The main idea is to control the oscillation of $\varphi$ between neighbouring cells, and use this to show that the effect of the cold boundary propagates inward along separatrices, all the way to the center cell. The techniques used are similar to~\cites{bblFannjiangKiselevRyzhik,bblKiselevRyzhik}. The main new (and non-trivial)
difficulty in our situation is that the number of cells also increases with the amplitude. This requires us to estimate the oscillation of $\varphi$ between cells in terms of energies localised to each cell (Proposition~\ref{ppnTotalOsc}, below). Here the assumption that $L$ is not too large comes into play. Finally, the key idea in the proof is to use a min-max argument (Lemma~\ref{lmaGI}, below) to show that $\varphi$ is small on the boundaries of all cells.

Moreover, once smallness on separatrices is established, our proof may be modified to show that under a stronger assumption 
\begin{equation}\label{eqnStrongFlowAssumption-bis}
\liminf_{A \to \infty} \frac{\sqrt{A}}{L^2 \log A \log L} =+\infty,
\end{equation}
we have a precise asymptotics
\begin{equation}\label{070802}
\lim_{A\to\infty}\lambda_{L,A} = 
\inf\left\{ \int_{Q} \abs{\grad w}^2 \;\middle|\; w \in H^1_0(Q), \int_{Q} w^2 = 1, \text{ and } w \cdot \grad v = 0 \right\},
\end{equation}
where $Q$ is a single cell. This is the same as the variational principle in~\cite{bblBeresHamelNadirshvilli}. We remark however that~\cite{bblBeresHamelNadirshvilli} only gives~\eqref{070802} for fixed $L$ as $A \to \infty$.\smallskip

Turning to the homogenization regime (Theorem~\ref{thmEvalHomog}), we remark first that homogenization of eigenvalues has not been as extensively studied as other homogenization problems. This is possibly because eigenvalues involve the infinite time horizon.  We refer to~\cites{AlCap,Kes1,Kes2,SV1,SV2} that all study self-adjoint problems for some results on the homogenization of the eigenvalues in oscillatory periodic media. The extra difficulties in the present paper come both from two sources. First, since the problem is not self adjoint, a variational principle for the eigenvalue is not available. Second, as $A$ and $L$ tend to $\infty$, we don't have suitable aprori bounds because either the domain is not compact, or the effective diffusivity is unbounded.

Our proof uses a multi-scale expansion to construct appropriate sub and super solutions. When $A$ is fixed, it usually suffices to consider a multi-scale expansion to the first corrector. However, in our situation, this is not enough, and we are forced to consider a multi-scale expansion up to the second corrector.

Of course an asymptotic profile, and explicit bounds are readily available~\cite{bblFannjiangPapanicolaou} for the first corrector. However, to the best of our knowledge, bounds on the second corrector as $A \to \infty$  have not been studied. There are two main problems to obtaining these bounds. The first problem is appearance of that terms involving the slow gradient of the second corrector multiplied by $A$. In general, we have no way of bounding these terms. Luckily, if we choose our slow profile to be quadratic, then these terms idnetically vanish and present no problem at all!

The second problem with obtaining bounds on the second corrector is that it satisfies an equation where the first order terms depend on $A$. So one would expect the bounds to also depend on $A$, which would be catastrophic in our situation. However, for elliptic equations with a \textit{divergence free} drift, we have apriori $L^p \to L^\infty$ estimates which are \textit{independent} of the drift~\cites{bblBerestyckiKiselevNovikovRyzhik,bblFannjiangKiselevRyzhik}. This, combined with an explicit knowledge of the first corrector, allows us to obtain bounds on the second corrector that decay when $A \ll L^4$.

The sub and super solutions we construct for eigenvalue problem are done through the expected exit time. Since these are interesting in their own right, we describe them below. Let $\tau = \tau_{L, A}$ be the solution of
\begin{equation}\label{eqnExitTimeProblem}
\left\{
\begin{gathered}
- \lap \tau + A v \cdot \grad \tau = 1 \quad \text{in } D\\
\tau = 0 \quad \text{on } \del D,
\end{gathered}\right.
\end{equation}
where $v$ and $D$ are as in \eqref{eqnEval}. Though we do not use any probabilistic arguments in this paper, it is useful to point out the connection between $\tau$ and diffusions. Let $X$ be the diffusion
\begin{equation}\label{071502}
dX_t = -Av(X_t) \, dt + \sqrt{2} dW_t
\end{equation}
where $W$ is a standard $2$-dimensional Brownian motion. It is well known that $\tau$ is the expected exit time of the diffusion $X$ from the domain $D$. Numerical simulations of three realizations of $X$ are shown in Figure~\ref{figdiff}. Note that for ``small'' amplitude ($A=L^3$), trajectories of $X$ behave similarly to those of the Brownian motion. For a ``large'' amplitude ($A=L^{4.5}$), trajectories of $X$ tend to move ballistically along the skeleton of the separatrices.
\begin{figure}[htb]
    \centering
    \subfigure[Small amplitude ($A = L^3$)]{\includegraphics[width=5cm]{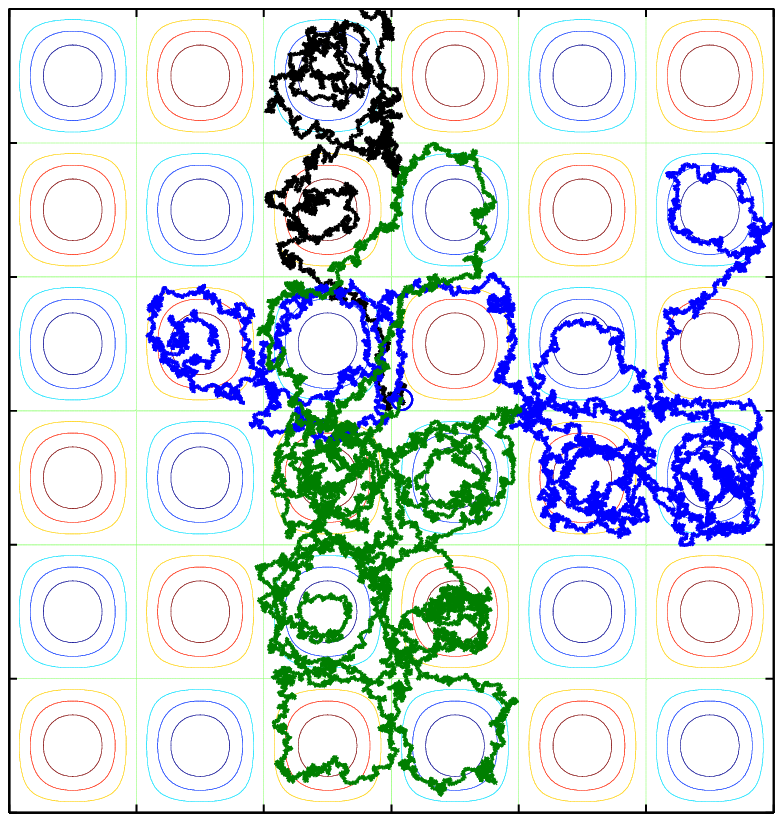}}
    \qquad
    \subfigure[Large amplitude ($A = L^{4.5}$)]{\includegraphics[width=5cm]{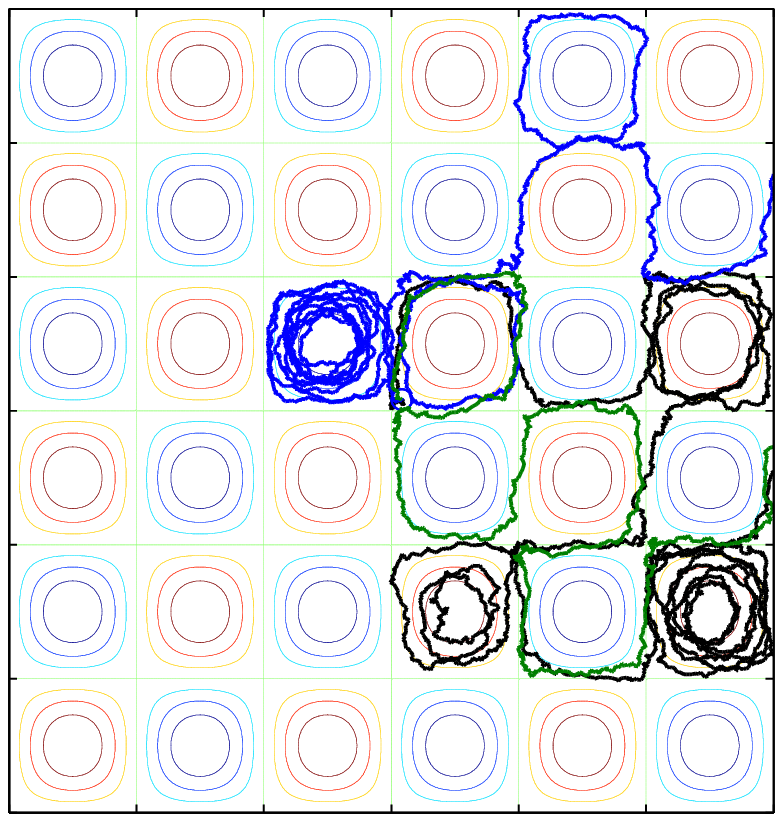}}
\caption{Trajectories of three realizations of the diffusion \eqref{071502}.}
\label{figdiff}
\end{figure}
\begin{figure}[htb]
    \centering
    \subfigure[Small amplitude ($A = L^3$)]{\includegraphics[height=5cm]{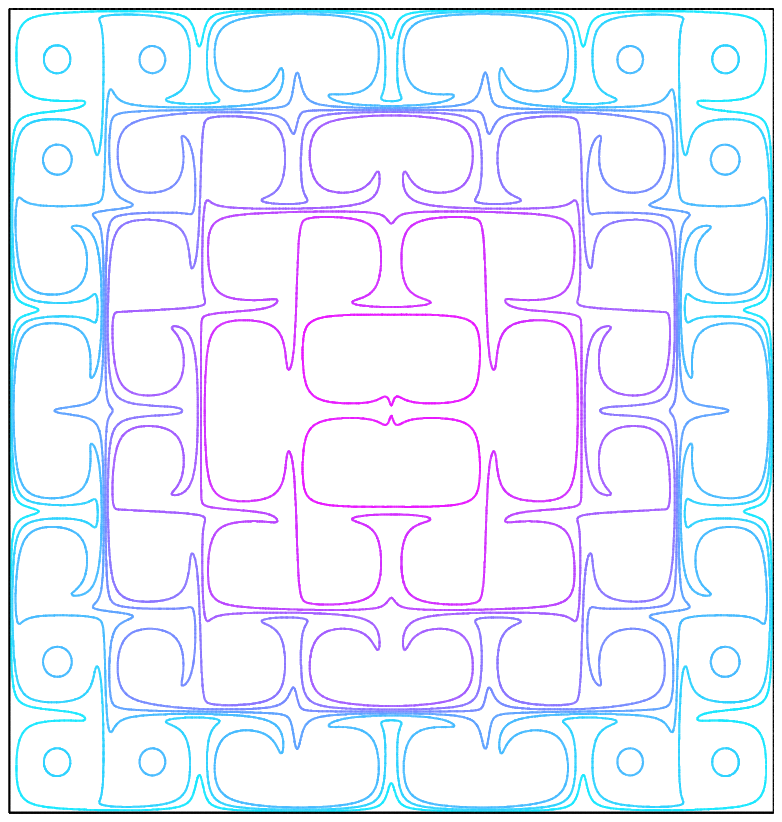}}
    \qquad
    \subfigure[Large amplitude ($A = L^5$)]{\includegraphics[height=5cm]{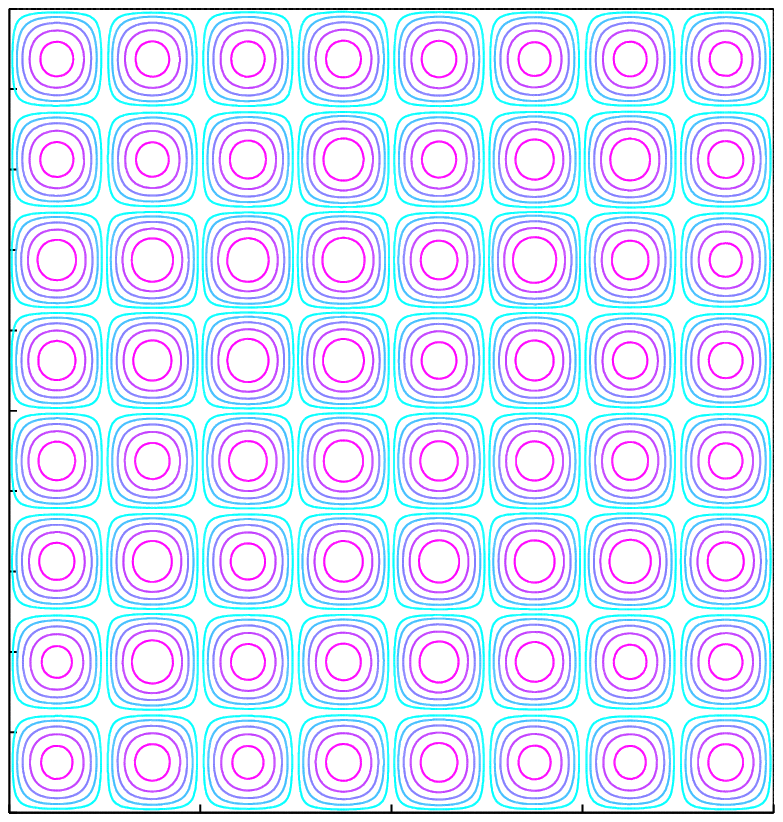}}
    \caption{A contour plot of $\tau(x,y)$.}
    \label{fgrtau}
\end{figure}

Similar to the eigenvalue problem, the behaviour of $\tau$ is described by two distinct regimes with a sharp transition. If $A \gg L^4$, then the stirring is strong enough to force the diffusion $X$ to exit $D$ almost immediately along separatrices. In this case, we show that $\tau \to 0$ on separatrices, and is bounded everywhere else above by a constant independent of $A$ and $L$. On the other hand, if $A \ll L^4$, then the stirring is not strong enough for the effect of the cold boundary to be felt in the interior. In this case, it takes the diffusion $X$ a very long time to exit from $D$, and $\tau \to \infty$ as $A, L \to \infty$. A numerical simulation showing $\tau$ in each of these regimes is shown in Figure~\ref{fgrtau}. The precise results are as follows.

\begin{theorem}[The averaging regime]\label{ppnExitTimeStrongFlow}  
Let $\tau = \tau_{L, A}$ be the solution to~\eqref{eqnExitTimeProblem}. Let $A \to \infty$, and suppose $L = L(A)$ varies such that~\eqref{eqnStrongFlowAssumption} is satisfied. There exists a constant $C$, independent of $A$, $L$, such that for all $A$ sufficiently large
$$
\tau(x)^2 \leq C \frac{L^2}{\sqrt{A}} \log A \log L, \quad\text{whenever } H(x) = 0.
$$
Consequently, if $H(x) = 0$, then $\tau (x) \to 0$ as $A \to \infty$, and $\norm{\tau}_{L^\infty(D)}$ is bounded uniformly in $A$.
\end{theorem}
\begin{theorem}[The homogenization regime] \label{ppnExitTimeHomog}
As with Theorem~\ref{ppnExitTimeStrongFlow}, let $\tau = \tau_{L,A}$ be the solution of~\eqref{eqnExitTimeProblem} on the square $D=[-L/2,L/2]\times[-L/2,L/2]$. Suppose now $L \to \infty$, and $A = A(L)$ varies such that~\eqref{eqnHomogAssumption} is satisfied, for some fixed $\alpha \in( 0,4)$. Then, for any $\delta > 0$, there exists a constant $C = C(\delta, \alpha, c) > 0$, independent of $A$, $L$, such that
\begin{equation}\label{071104}
C^{-1} \frac{L^2}{\sqrt{A}}\leq\tau(x) \leq C \frac{L^2}{\sqrt{A}},\quad \text{whenever } \abs*{x } \leq (1 - \delta) \frac{L}{2}
\end{equation}
for all $L$ sufficiently large. Consequently, $\tau \to \infty$ as $L \to \infty$, uniformly on compact sets.
\end{theorem}

The proof of Theorem~\ref{ppnExitTimeStrongFlow} is similar steps to that of Theorem~\ref{thmEvalStrongFlow}. For the proof of Theorem~\ref{ppnExitTimeHomog}, as mentioned earlier, we need to perform a multi-scale expansion up to two correctors, and choose the slow profile to be quadratic. When the domain is a disk, a quadratic function is exactly the solution to the homogenized problem! This gives us a sharper estimate for $\tau$.
\begin{proposition}\label{prop-disk}
Let $B_L$ be a disk of radius $L$, and $\tau_{A,L}$ be the solution of \eqref{eqnExitTimeProblem} in $B_L$. If $A$ and $L$ satisfy the assumptions in Theorem~\ref{ppnExitTimeHomog} then
\begin{equation}\label{eqnTau1Bound-bis}
\abs*{\tau(x) - \frac{1}{2 \trace(\bar\sigma(A))} \left( L^2 - \abs{x}^2 \right)} \leq c \frac{L}{A^{1/4}}
\end{equation}
where $c > 0$ is independent of $A$ and $L$. Here $\bar \sigma(A)$ is the effective diffusion matrix and $\trace(\bar\sigma(A))$ denotes the trace of this matrix.
\end{proposition}
\begin{remark}
Note that right hand side of\eqref{eqnTau1Bound-bis} tends to infinity as $A, L \to \infty$. However, by~\eqref{eqnEffectiveDiffusivityAsym} the terms on the left are of order $L^2/\sqrt{A}$, which dominates the right hand side. Thus~\eqref{eqnTau1Bound-bis} immediately implies~\eqref{071104}.
\end{remark}

By fitting a disk inside, and outside a square, Proposition~\ref{prop-disk} quickly implies Theorem~\ref{ppnExitTimeHomog}. Further, since it is well known that the principal eigenvalue is bounded below by the maximum expected exit time, the lower bound in Theorem~\ref{thmEvalHomog} also quickly follows from Proposition~\ref{prop-disk}. The upper bound is a little more technical, however, also uses Proposition~\ref{prop-disk} as the main idea.\medskip

We mention that we have chosen to use the particular form of the stream-function $H(x_1,x_2)=\frac{1}{\pi}\sin(\pi x_1)\sin(\pi x_2)$ simply 
for the sake of convenience. All our results may be generalized to other periodic flows with a cellular structure without any difficulty.
We also believe that for other flows the transition from the averaging to the homogenization regime happens when the effective
diffusivity $\bar \sigma(A)$ balances with the domain size $L$.  That is, when
the ``homogenized eigenvalue'' $\bar \sigma(A)/L^2$ is of the same order
as the ``strong flow'' eigenvalue:
\begin{equation}\label{072802}
\lim_{A\to\infty}\frac{\bar \sigma(A)}{L^2\lambda_{L,A}}\approx 1.
\end{equation}
We leave this question for a future study.

This paper is organized as follows. The averaging regime is considered in Sections~\ref{sec:eigfast} and~\ref{sec:exitfast}. The former contains the proof of Theorem~\ref{thmEvalStrongFlow} and the latter of Theorem~\ref{ppnExitTimeStrongFlow}. The rest of the paper addresses the homogenization regime. The key step here is Proposition~\ref{prop-disk} proved in Section~\ref{sxnHomogExitTime}. From this, Theorem~\ref{ppnExitTimeHomog} quickly follows, and the proof is presented the same section. Theorem~\ref{thmEvalHomog} is proved in Section~\ref{sec:homeig}.

\subsubsection*{Acknowledgement}
GI was supported by NSF grant DMS-1007914, TK by Polish Ministry of Science and Higher Education grant NN 201419139, AN by NSF grant DMS-0908011, and LR by NSF grant DMS-0908507, and NSSEFF fellowship. We thank Po-Shen Loh for suggesting the proof of Lemma~\ref{lmaGI}.
 
 \section{The eigenvalue in the strong flow regime}\label{sec:eigfast}

In this section we present the proof of Theorem~\ref{thmEvalStrongFlow}.
First, we discuss the proof of the upper bound in \eqref{eqnStrongFlowEvalBounds},
followed by the proof of the corresponding lower bound, and, finally, of the 
limiting behavior in \eqref{070802}.

\subsection{The upper bound}\label{sxnStrongFlowUpperBound}

The upper bound for $\lambda$ in \eqref{eqnStrongFlowEvalBounds}  follows 
directly from the techniques of~\cite{bblBeresHamelNadirshvilli}. 
We carry out the details below.
Following~\cite{bblBeresHamelNadirshvilli}, given any test function
$w \in H^1_0(D)$, and a number $\alpha > 0$, we multiply~\eqref{eqnEval} 
by $ {w^2}/{(\varphi + \alpha)}$ and integrate over $D$ to obtain
\begin{equation}\label{eqnLambdaUBD1}
\lambda \int_D \frac{w^2 \varphi}{\varphi + \alpha} = - \int_D \frac{w^2 \lap \varphi}{\varphi + \alpha} + A \int_D \frac{w^2}{\varphi + \alpha} v \cdot \grad \varphi.
\end{equation}
For the first term on the right, we have
\begin{eqnarray*}
&&- \int_D \frac{w^2 \lap \varphi}{\varphi + \alpha} 
= \int_D \grad \varphi \cdot \left( \frac{ 2 w (\varphi + \alpha) \grad w - 
w^2 \grad \varphi}{ (\varphi + \alpha)^2} \right)\\
&&= \int_D \abs{\grad w}^2 - \int_D \frac{\abs{w \grad \varphi 
- (\varphi + \alpha) \grad w}^2}{(\varphi + \alpha)^2}
\leq \int_D \abs{\grad w}^2.
\end{eqnarray*}
For the second term on the right of \eqref{eqnLambdaUBD1} we have, 
since $u$ is incompressible,
\begin{align*}
\int_D \frac{w^2}{\varphi + \alpha} v \cdot \grad \varphi = \int_D w^2 v \cdot \grad \ln(\varphi + \alpha) = -2 \int_D \ln(\varphi + \alpha) w (v \cdot \grad w).
\end{align*}
Hence, equation~\eqref{eqnLambdaUBD1} reduces to
\begin{equation}\label{eqnLambdaUBD2}
\lambda \int_D \frac{w^2 \varphi}{\varphi + \alpha} \leq \int_D \abs{\grad w}^2 - 
2A \int_D \ln(\varphi + \alpha) w (v \cdot \grad w).
\end{equation}

Now, choose $w$ to be any $H^1_0(D)$ first integral of $v$ (that is,
 $v \cdot \grad w = 0$). Then, equation~\eqref{eqnLambdaUBD2} reduces to
$$
\lambda \int_D \frac{w^2 \varphi}{\varphi + \alpha} \leq \int_D \abs{\grad w}^2.
$$
Upon sending $\alpha \to 0$, the Monotone Convergence Theorem shows
$$
\lambda \int_D w^2 \leq  \int_D \abs{\grad w}^2
$$
for any $H^1_0(D)$ first integral of $v$. Choosing $w = H(x)$, which, 
of course, does not depend on $L$, we immediately see that
\[ 
\lambda \leq \left( \int_D H^2  \right)\inv \int_D \abs{\grad H}^2
    =\left( L^2 \int_{Q_0} H^2  \right)\inv L^2 \int_{Q_0} \abs{\grad H}^2\\
    =\left( \int_{Q_0} H^2  \right)\inv \int_{Q_0} \abs{\grad H}^2,
\]
where $Q_0$ is any cell in $D$. This gives a finite upper bound for $\lambda$ that is independent of $L$ and $A$.

\subsection{The lower bound}\label{sxnStrongFlowLowerBound}
The outline of the proof is as follows. The basic idea is that if the domain size $L$ is 
not too large, and the flow is sufficiently strong,  the eigenfunction $\varphi$ 
should be small not only near the boundary $\partial D$
but also on the whole skeleton of separatrices inside $D$. Therefore, the Dirichlet
eigenvalue problem for the whole domain $D$ is essentially equivalent to a one-cell
Dirichlet problem, which gives the correct asymptotics for the eigenvalue for $A$ large.
To this end, we first estimate the oscillation of $\varphi$   
along a streamline of $v$ inside one cell that is sufficiently close to 
the separatrix, and show that this oscillation is small: see Lemma~\ref{lmaOscOnSep}. 
Next, we show that the difference of the values of $\varphi$ on two streamlines of $v$ 
(sufficiently close to the separatrix) in two neighbouring cells must be small, as in Lemma~\ref{lmaOscBetweenCells} below. These two steps are very similar to those in~\cite{bblFannjiangKiselevRyzhik}, and their proofs are only sketched.

Now, considering the `worst case scenario' of the above oscillation estimates, 
we obtain a pointwise upper bound on $\varphi$ on streamlines of $v$ near 
separatrices in terms of the principal eigenvalue $\lambda$, and 
$\norm{\varphi}_{L^2}^2$: see Lemma~\ref{lmaGI}. Next, we show that the streamlines 
above enclose a large enough region to encompass most of the mass of $\varphi^2$. Finally, we use the drift independent apriori estimates in~\cites{bblBerestyckiKiselevNovikovRyzhik,bblIyerNovkiovRyzhikZlatos} to obtain the desired lower bound on $\lambda$.

\subsubsection*{A streamline oscillation estimate}

The basic reason behind the fact that the eigenfunction is constant on streamlines is 
the following estimate, originally due to S. Heinze~\cite{Heinze}. 
\begin{lemma}\label{lmaHeinz}
There exists a constant $C>0$ so that   we have
\begin{equation}\label{070804}
\int_{D}|v\cdot \nabla\varphi|^2dx\leq\frac{C}{A}\int_{D}|\nabla\varphi|^2dx=\frac{C\lambda}{A}\|\varphi\|_{L^2}^2.
\end{equation}
\end{lemma}
\begin{proof}
Let us use the normalization $\|\varphi\|_{L^2}=1$. We multiply~\eqref{eqnEval} by $v \cdot \nabla\varphi$, 
and integrate over $Q_i$. This gives
$$
A \int_{D} \abs{v \cdot \grad \varphi}^2 = 
\lambda \int_{D} \varphi (v \cdot \grad \varphi) + 
\int_{D} \lap \varphi (v \cdot \grad \varphi).
$$
Notice that
$$
\int_{D} \varphi (v \cdot \grad \varphi) = 
\frac{1}{2} \int_{D} v \cdot \grad \left( \varphi^2 \right) =   0,
$$
since $v\cdot\nu = 0$ on $\del Q_i$ and $\nabla\cdot v=0$. Similarly, we have, 
as $v\cdot\nabla\varphi=0$ on $\partial D$:
\begin{multline*}
\int_{D} \lap \varphi (v \cdot \grad \varphi) = 
    -\sum_{j=1}^2\int_{D} \del_j \varphi \left( \del_j v \cdot \grad \varphi \right) - \sum_{j=1}^2\int_{D} \del_j \varphi \left( v \cdot \grad \del_j \varphi \right) + \int_{\partial D}(\nu\cdot\nabla\varphi)(v\cdot\nabla\varphi) \, dS\\
    \leq \norm{\grad v}_{L^\infty(Q_i)} \norm{\grad \varphi}_{L^2(D)}^2 - \frac{1}{2} \sum_{j=1}^2\int_{D} \divergence \left( v \left(\del_j \varphi\right)^2 \right)=C \lambda,
\end{multline*}
and consequently we obtain \eqref{070804}.
\end{proof}

The next lemma bounds locally the oscillation on streamlines in terms of the $L^2$-norm 
of $v\cdot\nabla\varphi$.
\begin{lemma}\label{lmaOscOnSep}
Let $Q_i$ be any cell. 
For any $\delta_0 > 0$, there exists $\Gamma_i \subset Q_i$ such that $\Gamma_i$ is a level set of $H$, $\abs{H(\Gamma_i)} \in (\delta_0, 2\delta_0)$, and
\begin{equation}\label{eqnOscOnSep1}
\sup_{x_1,x_2 \in \Gamma_i} \abs{\varphi(x_1) - \varphi(x_2)}^2 \leq C \frac{1}{\delta_0} \log\left( \frac{1}{\delta_0} \right) \int_{Q_i} \abs{v \cdot \grad \varphi}^2.
\end{equation}
for some constant $C$ independent of $A, L, \delta_0$. 
\end{lemma}
We will see that $\delta_0$ is the `width' of the boundary layer, and will eventually be chosen to be $\delta_0\approx {1}/{\sqrt{A}}$.

\begin{proof}
The proof is straightforward and similar bounds have already appeared in~\cites{bblKiselevRyzhik,bblFannjiangKiselevRyzhik}. We sketch the details here 
for convenience. First we introduce curvilinear coordinates in the cell $Q_i$. 
For this, let $(x_i, y_i)$ be the center of $Q_i$, and $\Theta_i$ be the solution of
\begin{equation}\label{eqnThetaDef}
\left\{
\begin{aligned}
\grad \Theta_i \cdot \grad H &= 0 && \text{in } Q_i - \{(x_i +t, y_i) \st t \geq 0\}\\
\Theta_i(x, y) &= \tan\inv\left( \frac{y - y_i}{x - x_i} \right) && \text{on } \del Q_i.
\end{aligned}\right.
\end{equation}
As usual, we extend $\Theta$ to $Q_i$ by defining it to 
be $0$ (or $2\pi$) on $\{(x_i + t, y_i) \st t \geq 0 \}$.

In the coordinates $(h, \theta)$ given by the functions $H$ and $\Theta_i$, 
it is easy to check that
$$
\frac{\del \varphi}{\del \theta} = \frac{v \cdot \grad \varphi}{\abs{\grad \Theta_i} \abs{\grad H}}.
$$
Assume, for simplicity, that $H \geq 0$ on $Q_i$. Then for any 
$h \in (\delta_0, 2\delta_0)$, we have
\begin{align*}
\sup_{\theta_1, \theta_2} \abs{\varphi(h, \theta_1) - \varphi(h,\theta_2)}^2
    &\leq \left( \int_{\{H = h\}} \abs{v \cdot \grad \varphi} \, 
\frac{d \theta}{\abs{\grad \Theta_i} \abs{\grad H}} \right)^2\\
    &\leq \int_{\{H = h\}} \abs{ v \cdot \grad \varphi}^2 \frac{d\theta}{\abs{\grad \Theta_i}\abs{\grad H}} \int_{\{H = h\}} \frac{d\theta}{\abs{\grad \Theta_i}\abs{\grad H}}\\
    &\leq C \ln \left(\frac{1}{\delta_0}\right) \int_{\{H = h\}} \abs{ v \cdot \grad \varphi}^2 \frac{d\theta}{\abs{\grad \Theta_i}\abs{\grad H}}.
\end{align*}
The last inequality follows from the fact that
$$
\int_{\{H = h\}} \frac{d\theta}{\abs{\grad \Theta_i}\abs{\grad H}}
    \leq C \ln\frac{1}{\delta_0},
$$
as the length element along the contour $dl=d\theta/|\nabla\Theta|$.
Now integrating over $(\delta_0, 2 \delta_0)$, we get
$$
\int_{\delta_0}^{2\delta_0} \sup_{\theta_1, \theta_2} \abs{\varphi(h, \theta_1) - \varphi(h,\theta_2)}^2 \, dh \leq C \ln \frac{1}{\delta_0} \int_{Q_i} \abs{v \cdot \grad \varphi}^2
$$
and~\eqref{eqnOscOnSep1} follows from the mean value theorem.
\end{proof}

\subsubsection*{Variation between neighboring cells}

Now, we consider two streamlines on which the solution is nearly constant and 
estimate the possible jump in the value of $\varphi$ between them.

\begin{lemma}\label{lmaOscBetweenCells}
Let $Q_i$ and $Q_j$ be two neighbouring cells, $\Gamma_i \subset Q_i$, $\Gamma_j \subset Q_j$ the respective level sets from Lemma~\ref{lmaOscOnSep}, and let $h_i = H(\Gamma_i)$, $h_j = H(\Gamma_j)$. Then there exists $x_i \in \Gamma_i$ and $x_j \in \Gamma_j$ such that
\begin{equation}\label{071506}
\abs*{ \varphi(x_i) - \varphi(x_j) }^2 \leq C \delta_0 \int_{Q_i \cup Q_j} \abs{\grad \varphi}^2+
C\frac{1}{\delta_0} \log\left( \frac{1}{\delta_0}\right)\int_{Q_i \cup Q_j} \abs{v\cdot\grad \varphi}^2.
\end{equation}
\end{lemma}
\begin{proof}
Assume again
for simplicity that $H \geq 0$ on $Q_i$, and $Q_i$ is to the left of $Q_j$. 
Then, using the local
curvilinear coordinates $(h,\theta)$ around the common boundary between the cells $Q_i$ and $Q_j$, as
in \eqref{eqnThetaDef}, we have
$$
\varphi(h_i, \theta) - \varphi( h_j, \theta ) = 
\int_{h_j}^{h_i} \frac{\del \varphi}{\del h} \, dh.
$$
Now, let $\delta_1 \in (0, \frac{\pi}{2})$ be fixed. 
In the region $\abs{h}\leq h_i$, and $\abs{\theta} \leq \delta_1$, 
we know that $\abs{\grad H} \approx 1$ and $\abs{\grad \Theta} \approx 1$. 
Hence, we have
$$
\int_{Q_i} \abs{\grad \varphi}^2 \geq C\int_{- \delta_1}^{ \delta_1}\int_{h_i}^{h_j}  \abs*{\frac{\del \varphi}{\del h}}^2 
\geq \frac{C}{\delta_0} \inf_{\abs{\theta} \leq \delta_1} \abs*{\varphi(h_j, \theta) - \varphi( h_i,  \theta )}^2.
$$
However, Lemma~\ref{lmaOscOnSep}  shows that
\begin{equation}\label{071504}
\sup_{\abs{\theta} \leq \delta_1} \abs*{\varphi(h_j, \theta) - \varphi( h_i,  \theta )}^2\leq
\inf_{\abs{\theta} \leq \delta_1} \abs*{\varphi(h_j, \theta) - \varphi( h_i,  \theta )}^2+
C \frac{1}{\delta_0} \log\left( \frac{1}{\delta_0} \right) \int_{Q_i\cup Q_j} \abs{v \cdot \grad \varphi}^2.
\end{equation}
This concludes the proof of \eqref{071506}.
\end{proof}

\subsubsection*{Variation between two far away cells}
For each cell $Q_i$,
we set
\[
\alpha_i=\int_{Q_i} \left( \abs{\nabla\varphi}^2 + A \abs{v \cdot \grad \varphi}^2 \right) \, dx
.
\]
Choosing $\delta_0 ={1}/{\sqrt{A}}$, Lemmas~\ref{lmaOscOnSep} 
and~\ref{lmaOscBetweenCells} immediately give the following oscillation estimate.
\begin{proposition}\label{ppnTotalOsc}
If $Q_i$ and $Q_j$ are any two cells, and $\Gamma_i \subset Q_i$, $\Gamma_j \subset Q_j$ the respective level sets from Lemma~\ref{lmaOscOnSep}, then
$$
\abs*{ \sup_{x_i \in \Gamma_i} \varphi(x_i) - \inf_{x_j \in \Gamma_j} \varphi(x_j)} 
    \leq C \frac{ (\log A)^{1/2} }{A^{1/4}} \sum_{\text{line}} \sqrt{\alpha_k}.
$$
where the sum is taken over any path of cells that connects $Q_i$ and $Q_j$, consisting of only horizontal and vertical line segments.
\end{proposition}

Lemma~\ref{lmaHeinz} implies that 
\begin{equation}\label{071002}
\sum_{j}\alpha_j\leq C\lambda\|\varphi\|_{L^2(D)}^2,
\end{equation}
with the summation taken over all cells in $D$.
Now, the key to the proof of the lower bound in Theorem~\ref{thmEvalStrongFlow} is to obtain an 
estimate on $\norm{\varphi}_{L^\infty(\Gamma_{i})}$ in terms of $(\sum_{i} \alpha_{i})^{1/2}$. 
A direct application of Cauchy-Schwartz to Proposition~\ref{ppnTotalOsc} is wasteful and does not yield a good enough estimate. 
What is required is a more careful estimate of the `worst case scenario' for the values of $\alpha_i$. 
This is the content of our next Lemma.
\begin{lemma}\label{lmaGI}
On any cell $Q_i$, we have
$$
\sup_{x \in \Gamma_i} \abs{\varphi(x)}^2
    \leq C \frac{\log A \log L}{A^{1/2}} \sum_{\text{all cells}} \alpha_i
    \leq C \frac{\log A \log L}{A^{1/2}} \lambda \norm{\varphi}_{L^2(D)}^2
$$
where $\Gamma_i \subset Q_i$ is the level set from Lemma~\ref{lmaOscOnSep}.
\end{lemma}
\begin{proof}
For notational convenience, in this proof only, we will assume  that 
$D = (-L - \frac{1}{2}, L + \frac{1}{2})^2$ is the square of side length $2L + 1$ centered $(0,0)$, 
and $Q_{i,j} = \{ (x,y) \st x \in [i - \frac{1}{2}, i + \frac{1}{2}), y \in [j - \frac{1}{2}, j+ \frac{1}{2})\}$ is the cell with center $(i,j)$.
Note that in the present proof we label the cells, (and contours $\Gamma_{ij}$ inside the cell $Q_{ij}$ we use below)
by  two indices that correspond to the
coordinates of the center of the cell.

Let $\mathfrak G_{i_0,j_0}$ denote the set of all paths of cells that join the boundary
$\del D$ to the cell $Q_{i_0,j_0}$ using only horizontal and vertical line segments. Let $\beta = (\beta_{i,j}) \in \R^{(2L+1)^2}$, $\beta_{i,j} \geq 0$, be a collection of non-negative
numbers assigned to each cell, and denote
\begin{equation}\label{eqnQ1}
q_{i_0,j_0}(\beta) = \min_{\mathfrak g\in \mathfrak G_{i_0, j_0}} \sum_{(i,j) \in \mathfrak g} \sqrt{\beta_{i,j}}.
\end{equation}
We first claim there exists an explicitly computable constant $C$, independent of $L, \beta, i_0, j_0$ such that
\begin{equation}\label{eqnQbd}
q_{i_0,j_0}(\beta)^2 \leq C  \log L\sum_{i,j = -L}^L   \beta_{i,j},
\end{equation}
which is an obvious improvement over the Cauchy-Schwartz estimate applied blindly to \eqref{eqnQ1}. This improvement
comes because we are taking the minimum over all such paths in \eqref{eqnQ1}.

To prove~\eqref{eqnQbd}, we define
$$
\qavg_{i_0, j_0}(\beta, \mathfrak G'_{i_0, j_0}) = \frac{1}{\abs{\mathfrak G'_{(i_0,j_0)}} } \sum_{\mathfrak g \in \mathfrak G'_{i_0,j_0}} \sum_{(i,j) \in \mathfrak g} \sqrt{\beta_{i,j}}.
$$
where $\mathfrak G'_{(i_0, j_0)}$ is any collection of (possibly repeated) paths in $\mathfrak G_{(i_0, j_0)}$. Since the minimum of a collection of numbers is not bigger than the average of any subset, we certainly have
$$
q_{i_0, j_0}(\beta) \leq \qavg_{i_0, j_0}(\beta, \mathfrak G'_{i_0, j_0})
$$
for any collection $\mathfrak G'_{i_0, j_0}$. The idea is to choose such a sub-collection in a convenient way.

We prove the claim for $(i_0, j_0) = (0,0)$. We choose $\mathfrak G'_{0,0}$ to consist of $(L+1)!$ paths, with the following 
property. All paths stay in the upper-right quadrant.
The last cell visited by all paths is $(0,0)$. The second to last cell visited by ${(L+1)!}/{2}$ paths (half of the collection)
is $(1,0)$, and the second to last cell visited by the remaining paths is $(0,1)$. Amongst the paths who's second to last cell is 
$(1,0)$, we choose $\mathfrak G'_{0,0}$ so that two thirds of these paths have $(2,0)$ as the third to last cell, and one third 
have $(1,1)$ as the third to last cell. Symmetrically, we choose $\mathfrak G'_{0,0}$ so that amongst all the paths who's second to
last cell is $(0,1)$, two thirds of these paths have $(0,2)$ as the third last cell, and one third have $(1,1)$ as the third to last cell.
Consequently exactly ${(L+1)!}/{3}$ paths have third to last cell $(2,0)$, exactly ${(L+1)!}/{3}$ paths in $\mathfrak G'_{0,0}$ 
have third to last cell $(1,1)$, and exactly ${(L+1)!}/{3}$ paths in $\mathfrak G'_{0,0}$ have third to last cell $(0,2)$.

Continuing similarly, we see that $\mathfrak G'_{0,0}$ can be chosen so that for any cell $(i,j)$ with $i + j \leq L$, 
\textit{exactly} ${(L+1)!}/{(i+j+1)}$ paths visit the cell $(i,j)$ as the $(i+j+1)^\text{th}$ to last cell. 
Finally, we assume that all paths in $\mathfrak G'_{0,0}$ start on the top boundary and proceed directly 
vertically downward until they hit a cell of the form $(k, L -k)$.

Let us count how many times each term $\sqrt{\beta_{i,j}}$ appears in the averaged sum 
$\qavg_{0,0}( \beta, \mathfrak G'_{0,0} )$. Clearly, if $i + j \leq L$, then the cell $(i,j)$ appears in exactly 
$\frac{(L+1)!}{i + j + 1}$ paths in $\mathfrak G'_{0,0}$.  On the other hand, if $i + j > L$, then the cell $(i,j)$ appears 
exactly $\frac{(L+1)!}{L + 1}$ paths. Consequently, we have
\begin{equation}\label{071212}
\qavg_{0,0}( \beta, \mathfrak G'_{0,0}) = \sum_{k=0}^{L-1} \frac{1}{k+1} \sum_{i = 0}^{k} \sqrt{\beta_{i, k - i}} +
 \frac{1}{L+1}  \sum_{i=0}^L \sum_{j=L-i}^L \sqrt{\beta_{i,j}}.
\end{equation}
We now maximize the sum in \eqref{071212} with the constraint
\begin{equation}\label{071214}
\sum_{i,j=0}^L\beta_{ij}=\sigma.
\end{equation}
Let $S$ denote the right side of \eqref{071212}, then at the maximizer of $S$ we have
\[
\frac{\del S}{\del\beta_{ij}}=\frac{1}{2(L+1)\sqrt{\beta_{ij}} },~~\hbox{ for $i+j> L$},
\]
and
\[
\frac{\del S}{\del\beta_{ij}}=\frac{1}{2(i+j+1)\sqrt{\beta_{ij}} },~~\hbox{ for $i+j\leq L$}.
\]
The Euler-Lagrange equations now imply that
\[
\beta_{ij}=\frac{\gamma}{4(L+1)^2}~~\hbox{ for $i+j> L$},
\]
and
\[
\beta_{ij}=\frac{\gamma}{4(i+j+1)^2}~~\hbox{ for $0\leq i+j\leq L$}.
\]
Here $\gamma$ is the Lagrange multiplier that can be computed from the constraint \eqref{071214}:
\[
\frac{\gamma(L+1)(L+2)}{4(L+1)^2}+\gamma\sum_{j=0}^{L-1}\frac{(j+1)}{4(j+1)^2}=\sigma.
\]
It follows that 
\[
\gamma=\sigma \bar\gamma(L),
\]
with $\bar\gamma(L)=O(1/\log L)$ as $L\to+\infty$. Hence, for the maximizer we get
\[
S\leq \frac{C\sqrt{\sigma}}{\sqrt{\log L}}\sum_{k=0}^{L-1}\frac{1}{k+1}+C\sqrt{\gamma}=C\sqrt{\sigma\log L},
\]
and thus \eqref{eqnQbd} holds for $i_0,j_0=(0,0)$. However, it is immediate to see that 
the previous argument can be applied to any cell considering appropriate collection of paths that say up and to the right of $(i_0,j_0)$,
whence \eqref{eqnQbd} holds for all $(i_0,j_0)$.

With \eqref{eqnQbd} in hand, we observe that Proposition~\ref{ppnTotalOsc} implies
\[
\norm{ \varphi }_{L^\infty(\Gamma_{i_0,j_0})}^2 \leq C \frac{\log A}{\sqrt{A}} q_{i_0,j_0}(\alpha)^2 
\leq C \frac{\log A}{\sqrt{A}} \log L \sum_{i=-L}^L \sum_{j=-L}^L \alpha_{i,j}
\leq C \frac{\log A \log L}{\sqrt{A}} \lambda \norm{\varphi}_{L^2}^2,
\]
where the last inequality follows from Lemma~\ref{lmaHeinz}. This concludes the proof.
\end{proof}

Our next step shows that the mass of $\varphi^2$ in the regions enclosed by the level sets $\Gamma_i$ is comparable to $\norm{\varphi}_{L^2(D)}^2$.
\begin{lemma}\label{lmaGI2}
Let $Q_i$ be a cell, and $\Gamma_i \subset Q_i$ the level set from Lemma~\ref{lmaOscOnSep}. Let $h_i = H(\Gamma_i)$, and $S_i = Q_i \cap \{\abs{H} < \abs{h_i} \}$ be a neighbourhood of $\del Q_i$. Let $Q_i' = Q_i - S_i$. Then, for $A$ sufficiently large, we have
\begin{equation}\label{071302}
\sum_i \norm{\varphi}_{L^2(Q_i')}^2 \geq \frac{1}{2} \norm{\varphi}_{L^2(D)}^2
\end{equation}
\end{lemma}
\begin{proof}
For any cell $Q_i$, the Sobolev restriction theorem  shows
\begin{equation*}
\int_{H = h} \abs{\varphi(h, \theta)}^2 \, \frac{d\theta}{\abs{\grad \Theta}} = \norm{\varphi}_{L^2(H\inv(h) \cap Q_i)}^2
\leq C \norm{\varphi}_{H^1(Q_i)}^2 = C\left( \alpha_i + \norm{\varphi}_{L^2(Q_i)}^2 \right).
\end{equation*}
Thus, using curvilinear coordinates with respect to the cell $Q_i$, and assuming, for simplicity, that $h_i = H(\Gamma_i) > 0$, gives
\begin{align*}
\norm{\varphi}_{L^2(S_i)}^2
    &= \int_{h = 0}^{h_i} \int_{\theta = 0}^{2\pi} \varphi(h, \theta)^2 \frac{1}{\abs{\grad \Theta} \abs{\grad H}} \,  d\theta \,dh\\
    &\leq \int_{h=0}^{2 \delta_0} \left( \int_{\theta = 0}^{2\pi} \varphi(h, \theta)^2 \frac{1}{\abs{\grad \Theta}} \, d\theta\right) \left(\sup_{\theta \in [0, 2\pi]} \frac{1}{\abs{\grad H}}\right) \, dh\\
    &\leq C\left( \alpha_i + \norm{\varphi}_{L^2(Q_i)}^2 \right) \int_{h=0}^{2 \delta_0} \frac{1}{\sqrt{h}} \, dh= C \sqrt{\delta_0} \left( \alpha_i + \norm{\varphi}_{L^2(Q_i)}^2 \right).
\end{align*}
Summing over all cells gives
\[
\sum_i \norm{\varphi}_{L^2(S_i)}^2 \leq C \sqrt{\delta_0} \left( \norm{\grad \varphi}_{L^2(D)}^2 + \norm{\varphi}_{L^2(D)}^2 \right)
\leq C \sqrt{\delta_0} \left( 1 + \lambda \right) \norm{\varphi}_{L^2(D)}^2 \leq C \sqrt{\delta_0}  \norm{\varphi}_{L^2(D)}^2
\]
where the last inequality follows using the upper bound in~\eqref{eqnStrongFlowEvalBounds} which was proved in Section~\ref{sxnStrongFlowUpperBound}. Since $\delta_0 \to 0$ as $A \to \infty$, and
$$
\norm{\varphi}_{L^2(D)}^2 = \sum_i \norm{\varphi}_{L^2(S_i)}^2 + \sum_i \norm{\varphi}_{L^2(Q_i')}^2,
$$
inequality \eqref{071302} follows.
\end{proof}

Our final ingredient is a drift \textit{independent} $L^p \to L^\infty$ estimate in~\cite{bblBerestyckiKiselevNovikovRyzhik}. We recall it here for convenience.
\begin{lemma}[Lemma 1.3 in \cite{bblBerestyckiKiselevNovikovRyzhik}]\label{lmaDriftIndLinfBound}
Let $\Omega \subset \R^d$ be a domain, $w$ be divergence free, and $\theta$ be the solution to
\begin{equation*}
\left\{
\begin{aligned}
-\lap \theta + w \cdot \grad \theta &= f &&\text{in } \Omega\\
\theta &= 0 &&\text{on }\del \Omega,
\end{aligned}\right.
\end{equation*}
with $f \in L^p(\Omega)$ for some $p > d$. There exists a constant $c = c(\Omega, d, p) > 0$, \emph{independent of $w$}, such that $\norm{\theta}_{L^\infty} \leq c \norm{f}_{L^p}$.
\end{lemma}

We are now ready to prove the lower bound in Theorem~\ref{thmEvalStrongFlow}.
\begin{proof}[Proof of the lower bound in Theorem~\ref{thmEvalStrongFlow}]
Using the notation from Lemma~\ref{lmaGI2}, define $D' = \bigcup_i Q_i'$, and let $Q_j$ be a cell such that $\norm{\varphi}_{L^\infty(Q_j')} = \norm{\varphi}_{L^\infty(D')}$. 
Then
\begin{equation}\label{eqnL2DPrime}
\norm{\varphi}_{L^2(D')}^2 = \sum_i\norm{\varphi}_{L^2(Q_i')}^2  \leq L^2 \norm{\varphi}_{L^\infty(Q_j')}^2,
\end{equation}
and it follows from Lemmas~\ref{lmaGI} and~\ref{lmaDriftIndLinfBound} that
\begin{align*}
\norm{\varphi}_{L^\infty(Q_j')}
    &\leq C \left( \lambda \norm{\varphi}_{L^\infty(Q'_i)} + \norm{\varphi}_{L^\infty(\Gamma_j)} \right)\\
    &\leq C \left( \lambda \norm{\varphi}_{L^\infty(Q_j')} + \frac{1}{A^{1/4}} (\log A \log L)^{1/2} \sqrt{\lambda} \norm{\varphi}_{L^2(D)} \right)\\
    &\leq C \left( \lambda \norm{\varphi}_{L^\infty(Q_j')} + \frac{1}{A^{1/4}} (\log A \log L)^{1/2} \sqrt{\lambda} \norm{\varphi}_{L^2(D')} \right)
\end{align*}
where the last inequality follows from Lemma~\ref{lmaGI2}. Consequently,
$$
\lambda \geq \frac{1}{2C} \quad\text{or}\quad \lambda \geq 
\frac{\norm{\varphi}_{L^\infty(Q_j')}^2}{ \norm{\varphi}_{L^2(D')}^2} \frac{A^{1/2}}{2C \log A \log L} \geq 
\frac{A^{1/2}}{2C L^2 \log A \log L}
$$
where the last inequality follows from equation~\eqref{eqnL2DPrime}. This proves the lower bound on 
$\lambda$ in Theorem~\ref{thmEvalStrongFlow}.
\end{proof}

\section{The exit time in the strong flow regime}\label{sec:exitfast}
In this section we sketch the proof of Theorem~\ref{ppnExitTimeStrongFlow}. The techniques in Section~\ref{sxnStrongFlowLowerBound} readily show that oscillation of $\tau$ on stream lines of $v$ becomes small. 
Now, the key observation in the proof of Theorem~\ref{ppnExitTimeStrongFlow} is an \textit{explicit, drift independent} 
upper bound on the exit time. We state this below.

\begin{lemma}[Theorem~1.2 in \cite{bblIyerNovkiovRyzhikZlatos}]\label{lmaDriftIndTauBound}
Let $\Omega\subset \mathbb{R}^d$ be a bounded, piecewise $C^1$ domain, and $u:\Omega \to \R^n$ a $C^1$ divergence free vector field tangential to $\del \Omega$. Let $\tau'$ be the solution to
\begin{equation}\label{eqnTauPrime}
\left\{
\begin{gathered}
- \lap \tau' + u \cdot \grad \tau' = 1 \quad \text{in } \Omega\\
\tau' = 0 \quad \text{on } \del \Omega,
\end{gathered}\right.
\end{equation}
Then for any $p\in[1,\infty]$,
$$
\norm{\tau'}_{L^p(\Omega)} \leq \norm{\tau'_r}_{L^p(B)},
$$
where $B \subset \R^n$ is a ball with the same Lebesgue measure as $\Omega$, and $\tau'_r$ is the (radial, explicitly computable) solution to~\eqref{eqnTauPrime} on $B$ with $u \equiv 0$.
\end{lemma}

With this, we present the proof of Theorem~\ref{ppnExitTimeStrongFlow}.

\begin{proof}[Proof of Theorem~\ref{ppnExitTimeStrongFlow}]
Following the same method as that in Section~\ref{sxnStrongFlowLowerBound}, we obtain (analogous to Lemma~\ref{lmaGI})
\begin{equation}\label{eqnTauBound1}
\sup_{x \in \Gamma_i} \abs{\tau(x)}^2 \leq \frac{C \log A \log L}{\sqrt{A}} \sum_{\text{all cells}} \alpha_i,
\end{equation}
with $\alpha_i$ now equal $\alpha_i = \int_{Q_i} \abs{\grad \tau}^2$, where $Q_i$ is the $i^\text{th}$ cell. The sets $\Gamma_i \subset Q_i$ appearing in~\eqref{eqnTauBound1} are level sets of $H$ on which the oscillation of $\tau$ is small (analogous to Lemma~\ref{lmaOscOnSep}).

Now observe that
$$
\sum_{\text{all cells}} \alpha_i = \int_D \abs{\grad \tau}^2 = \int_D \tau,
$$
and so~\eqref{eqnTauBound1} reduces to
\begin{equation}\label{eqnTauBound2}
\sup_{x \in \Gamma_i} \abs{\tau(x)}^2 \leq \frac{C \log A \log L}{\sqrt{A}} \int_D \tau.
\end{equation}

Letting $Q_i'$ be the region enclosed by $\Gamma_i$, we obtain (similar to Lemma~\ref{lmaGI2})
\begin{equation}\label{eqnIntQijPrimeTau}
\int_D \tau \leq 2 \sum_{\text{all cells}} \int_{Q_i'} \tau
\end{equation}
for large enough $A$. By Lemma~\ref{lmaDriftIndTauBound} we see
$$
\int_{Q'_i} \tau \leq C \left(1 + \norm{\tau}_{L^\infty(\Gamma_i)}\right)
$$
and hence
$$
\int_D \tau \leq 
    C L^2 + C L^2 \frac{(\log A \log L)^{1/2}}{A^{1/4}} \left( \int_D \tau \right)^{1/2}.
$$
Solving the above inequality quickly yields
$$
\int_D \tau \leq C \left( L^2 + \frac{L^4}{\sqrt{A}} \log A \log L \right) \leq C L^2,
$$
where the second inequality above follows from the assumption~\eqref{eqnStrongFlowAssumption}. Substituting this in~\eqref{eqnTauBound2} immediately shows that
$$
\norm{\tau}_{L^\infty(\Gamma_i)}^2 \leq C \frac{L^2}{\sqrt{A}} \log A \log L.
$$

Now, to conclude the proof, we appeal to Lemma~\ref{lmaDriftIndTauBound} again. Let $S = D - \cup_i Q_i'$ be the (fattened) skeleton of the separatrices. Observe that $\abs{S} \leq C \frac{L^2}{\sqrt{A}}$ which, by assumption~\eqref{eqnStrongFlowAssumption}, remains bounded uniformly in $A$. Consequently, by Lemma~\ref{lmaDriftIndTauBound},
$$
\norm{\tau}_{L^\infty(S)} \leq C\abs{S}^2 + \norm{\tau}_{L^\infty(\del S)} \leq  C \frac{L^2}{\sqrt{A}} \log A \log L,
$$
which immediately yields the desired result.
\end{proof}

\section{The exit time in the homogenization regime}\label{sxnHomogExitTime}

\subsection*{Exit time from a disk.}

The key step in our analysis in the homogenization regime is Proposition~\ref{prop-disk}, and we begin with it's proof. The idea of the proof is to construct good sub and super solutions for the exit time problem in a disk of radius one. Let $\tau$ be the solution of~\eqref{eqnExitTimeProblem} in a ball of radius $L$. Let $B_1$ be a ball of radius $1$, and let~$\tau_1(x) = \tau(Lx)/ L^2$. Then $\tau_1$ is a solution of the PDE
\begin{equation}\label{eqnTau1}
\left\{\begin{aligned}
-\Delta\tau_1+ALv(Lx)\cdot\nabla\tau_1 &= 1	&&\text{in } B_1,\\
\tau_1 & =0 &&\text{on } \del B_1.
\end{aligned}\right.
\end{equation}
We begin by constructing an approximate solution $\tilde\tau_{1}$, by defining
\begin{equation}\label{eqnVarphiDefbis}
\tilde\tau_{1}(x) = \tau_{10}(x) + \frac{1}{L} \tau_{11}(x, y) + \frac{1}{L^2} \tau_{12}(y),
\end{equation}
where $y = Lx$ is the `fast variable'. We define $\tau_{10}$ explicitly by
\begin{equation}\label{eqnVarphi0bis}
\tau_{10}(x) = \frac{1-|x|^2}{2},
\end{equation}
and obtain equations for 
$\tau_{11}$ and $\tau_{12}$ using the standard periodic homogenization
multi-scale expansion. Using the identities
$$
\grad = \grad_x + L \grad_y
\quad\text{and}\quad
\lap = \lap_x + 2 L \grad_x \cdot \grad_y + L^2 \lap_y
$$
we compute
\begin{align*}
-\lap\tilde{\tau} + ALv \cdot \grad \tilde{\tau} =
    &-\lap_x \tau_{10} + AL v \cdot \grad_x \tau_{10}\\
    &+ \frac{1}{L} \bigl(
	\begin{multlined}[t]
	    -\lap_x \tau_{11}  - 2 L \grad_x \cdot \grad_y \tau_{11} - L^2 \lap_y \tau_{11}\\
	    + L A v \cdot \grad_x \tau_{11} + L^2 A v \cdot \grad_y \tau_{11} \bigr)
	\end{multlined}\\
    &+ \frac{1}{L^2} \left( - L^2 \lap_y \tau_{12} + L^2 A v \cdot \grad_y \tau_{12} \right).
\end{align*}
We choose $\tau_{11}$ to formally balance the $O(L)$ terms. That is, 
we define $\tau_{11}$ to be the mean-zero, periodic function such that
\begin{equation}\label{eqnVarphi1bis}
-\lap_y \tau_{11} + A v \cdot \grad_y \tau_{11} = - A v(y) \cdot \grad_x \tau_{10}.
\end{equation}
We clarify that when dealing with functions of the fast variable, we say that a function 
$\theta$ is periodic if $\theta( y_1 + 2, y_2) = \theta(y_1, y_2 + 2) = 
\theta(y_1, y_2)$ for all $(y_1, y_2) \in \R^2$. This is because our drift $v$ is periodic, with period $2$ in the fast variable, and each cell is a square of side length 2, in the fast variable.

Now we choose $\tau_{12}$ to formally balance the $O(1)$ terms. Define $\tau_{12}$ 
to be the mean-zero, periodic function such that
\begin{equation}\label{eqnVarphi2bis}
-\lap_y \tau_{12} + A v \cdot \grad_y \tau_{12} = 2 \grad_x \cdot \grad_y \tau_{11} - A \left( v \cdot \grad_x \tau_{11} - \average{v \cdot \grad_x \tau_{11}}\right),
\end{equation}
where $\average{\cdot}$ denotes the mean with respect to the fast variable~$y$. Observe that
we had to introduce the term $A \average{v \cdot \grad_x \tau_{11}}$ above to ensure that the right hand side is mean zero, to satisfy the compatibility condition.

We write
\begin{equation}\label{eqnVarphi1primebis}
\tau_{11}(x,y) =  \chi_1(y) \del_{x_1} \tau_{10}(x) + \chi_2(y)\del_{x_2}\tau_{10}(x)= - \chi_1(y)x_1 - \chi_2(y)x_2,
\end{equation}
where $\chi_j= \chi_j(y)$, $j=1,2$ are the mean zero, periodic solutions to
\begin{equation}\label{eqnChi1bis}
- \Delta_y \chi_j + A v \cdot \nabla_y \chi_j =  -A v_j.
\end{equation}
Using this expression for $\tau_{11}$ and~\eqref{eqnVarphi0bis} we 
simplify~\eqref{eqnVarphi2bis} to
\begin{equation}\label{eqnVarphi2primebis}
- \Delta_y \tau_{12} + A v \cdot \nabla_y \tau_{12} = - 2\del_{y_1} \chi_1 - 2\del_{y_2}\chi_2 + A (v_1 \chi_1  +v_2\chi_2 - \langle v_1  \chi_1 \rangle-\langle v_2  \chi_2 \rangle).
\end{equation}

The key observation is that with our choice of $\tau_{10}$, the right side of \eqref{eqnVarphi2primebis} 
is \textit{independent} of the slow variable. Our aim is to show that  $\tilde{\tau}$  satisfies the estimates~\eqref{eqnSuperSolBoundbis} and~\eqref{eqnSuperSolbis} below.
\begin{lemma}\label{lmaSupSolbis}
There exists a positive constant $c_0 = c_0(\alpha)$ 
independent of $A$, $L$, such that for~$\tilde{\tau}$ defined by~\eqref{eqnVarphiDefbis} 
we have
\begin{equation}\label{eqnSuperSolBoundbis}
\left|\tilde{\tau}(x)-\tau_{10}(x)\right| \leq c_0 L^{-\alpha/4} \qquad\text{for } x \in B_1
\end{equation}
and
\begin{equation}\label{eqnSuperSolbis}
 -\lap\tilde{\tau} + A L v(Lx) \cdot \grad_x \tilde{\tau} = \trace( \bar \sigma(A)).
\end{equation}
Here $\bar \sigma(A)$ is the effective diffusion matrix, given by~\eqref{intro10}.
\end{lemma}
We first use the Lemma to finish the proof of Proposition~\ref{prop-disk}.

\begin{proof}[Proof of  Proposition~\ref{prop-disk}]
The key observation we obtain from Lemma~\ref{lmaSupSolbis} is that, except for the boundary condition, the function
$\tilde{\tau}'(x)=\tilde{\tau}(x)/\trace( \bar \sigma(A) )$ satisfies exactly \eqref{eqnTau1}. This is a miracle that happens only when the domain is a disk. Then we get sub- and super-solutions for $\tau_1(x)$ by setting 
\[
\overline{\tau}(x)=\frac{1}{\trace(\bar \sigma(A))}\left[\tilde{\tau}(x)+\frac{2c_0}{L^{\alpha/4}}\right],
\]
and
\[
\underline{\tau}(x)=\frac{1}{\trace(\bar \sigma(A))}\left[\tilde{\tau}(x)-\frac{2c_0}{L^{\alpha/4}}\right].
\]
Lemma~\ref{lmaSupSolbis} implies that
\[
-\Delta\overline{\tau}(x)+ALv(Lx)\cdot\nabla\overline{\tau}(x)=-\Delta\underline{\tau}(x)+ALv(Lx)\cdot\nabla\underline{\tau}(x)=1.
\]
Further, since  $\tau_{10}(x)=0$ on $\del B_1$, equation~\eqref{eqnSuperSolBoundbis} implies that $\overline{\tau}(x)>0$ and $\underline{\tau}(x)<0$ on $\del B_1$. Consequently, $\overline{\tau}$ is a super solution, and $\underline{\tau}$ is a sub solution of~\eqref{eqnTau1}, and hence
\begin{equation}\label{eqnTau1B1}
\frac{1}{\trace(\bar \sigma(A))}\left[\tilde{\tau}(x)-\frac{2c_0}{L^{\alpha/4}}\right]\leq\tau_1(x)\leq 
\frac{1}{\trace(\bar \sigma(A))}\left[\tilde{\tau}(x)+\frac{2c_0}{L^{\alpha/4}}\right].
\end{equation}
Rescaling to the ball of radius $L$, we see
$$
\abs*{\tau(x) - \frac{L^2}{\trace( \bar\sigma(A))} \tau_{10}\left(\frac{x}{L}\right)} \leq \frac{L^2}{\trace(\bar\sigma(A))} \frac{4 c_0}{L^{\alpha / 4}}.
$$
Now using~\eqref{eqnEffectiveDiffusivityAsym} and~\eqref{eqnHomogAssumption} we obtain~\eqref{eqnTau1Bound-bis}.
\end{proof}

It remains to prove Lemma~\ref{lmaSupSolbis}.
\begin{proof}[Proof of Lemma~\ref{lmaSupSolbis}]
By our definition of $\tau_{10}$, $\tau_{11}$, $\tau_{12}$, we have
$$
- \lap \tilde{\tau} + A L v \cdot \grad \tilde{\tau}
    = -\lap_x \tau_{10} - \frac{1}{L} \lap_x \tau_{11} - A [\average*{v_1 \chi_1}+\average*{v_2 \chi_2}]
    = 2 + \average*{\abs{\grad \chi_1}^2 + \abs{\grad\chi_2}^2  } = \trace(\bar\sigma(A)),
$$
where the second inequality follows from~\eqref{eqnChi1bis}. This is exactly~\eqref{eqnSuperSolbis}.

To prove~\eqref{eqnSuperSolBoundbis}, we will show
\begin{equation}\label{eqnCorrectionTermBoundsbis}
\frac{1}{L} \norm{\tau_{11}}_{L^\infty} + \frac{1}{L^2} \norm{\tau_{12}}_{L^\infty} \leq c_0 L^{-\alpha/4},
\end{equation}
for some constant $c_0= c_0(\alpha)$, independent of $A$ and $L$.  
We will subsequently adopt the convention that $c$ is a constant, depending only on $\alpha$, which can change from line to line.

We first bound $\tau_{11}$. Let $Q = (-1, 1)^2$ be the fundamental domain of the fast variable. Let $\del_v Q$ and $\del_h Q$ denote the vertical and horizaondal boundaries of $Q$ respectively. Since $\chi_1(y_1, y_2)$ is odd in $y_1$ and even in $y_2$, by symmetry we have $\chi_1 = 0$ on $\del_v Q$, and $\del_{y_2} \chi_1 = 0$ on $\del_h Q$. Now if we consider the function $\chi_1 + y_1$, we have
\begin{gather*}
-\lap_y (\chi_1 + y_1) + A v \cdot \grad_y (\chi_1 + y_1) = 0,\\
\abs{\chi_1(y) + y_1} \leq 1 \text{ on } \del_v Q,
\quad\text{and}\quad
\frac{\del}{\del n} \left( \chi_1 + y_1 \right) = 0 \text{ on } \del_h Q.
\end{gather*}
Thus the Hopf Lemma implies $\chi_1 + y_1$ does not attain it's maximum on $\del_h Q$, except possibly at corner points. So by the maximum principle $\chi_1 + y_1$ attains its maximum on $\del_v Q$, and so
$$
\norm{\chi_1}_{L^\infty} \leq 1.
$$
Since $\chi_2$ is bounded similarly, we immediately have
\begin{equation}\label{eqnVarphi1Boundbis}
\frac{\norm{\tau_{11}}_{L^\infty}}{L} \leq c L^{-1} \leq c L^{-\alpha/4}.
\end{equation}

The last step is to prove a bound on $\norm{\tau_{12}}_{L^\infty}$. The crucial idea to bound $\tau_{12}$ is to split the right hand side of \eqref{eqnVarphi2primebis} into terms which are small in $L^p$, and terms which can be absorbed by the convection term. To this end, write $\tau_{12} = \eta + \psi_1 + \psi_2$ where $\eta$, $\psi_i$ are mean-zero, periodic solutions to
\begin{gather*}
- \Delta_y \eta + A v \cdot \nabla_y \eta =  -2 \sum_{i=1}^2 \del_{y_i} \chi_i\\
-\lap_y \psi_1 + A v \cdot \grad_y \psi_1 = A \sum_{i=1}^2 \left[ v_i \left(\chi_i + y_i - \frac{1}{2} \sign(y_i) \right) - \average{v_i \chi_i} \right]  \\
-\lap_y \psi_2 + A v \cdot \grad_y \psi_2 = -A \sum_{i=1}^2 v_i \left( y_i  - \frac{1}{2} \sign(y_i) \right).
\end{gather*}

Before estimating each term individually, we pause momentarily to explain this decomposition of $\tau_{12}$. The equation for $\eta$ is of course natural. 
The equation for $\psi_1$ stems from the well known behaviour of the corrector $\chi_1$. 
We know from~\cites{bblFannjiangPapanicolaou,bblNovikovPapanicolaouRyzhik,bblGorbNamNovikov} 
that $\chi_1$ grows rapidly in a boundary layer of width
$O({1}/{\sqrt{A}})$ and decreases \textit{linearly} in the cell interior. That is, for $i = 1,2$,
\[
\chi_i\approx \frac{1}{2}\sign y_i-y_i,
\]
away from the boundary layer. Further, by symmetry, $\chi_i$ is odd in $y_i$, and even in the other variable. Thus, we expect the term $\chi_i + y_i - \frac{1}{2} \sign(y_i)$   to be away from zero only in 
the boundary layer (see Figure~\ref{fgrxi}), and hence should have a small $L^p$ norm! Now the equation for 
$\psi_2$ is chosen to balance the remaining terms, and thankfully the right 
hand side can be absorbed in the convection term.
\begin{figure}[htb]
\centering
\subfigure[A $3D$ plot of the function $y_1+\chi_1(y_1,y_2)$.]{\includegraphics[height=7cm]{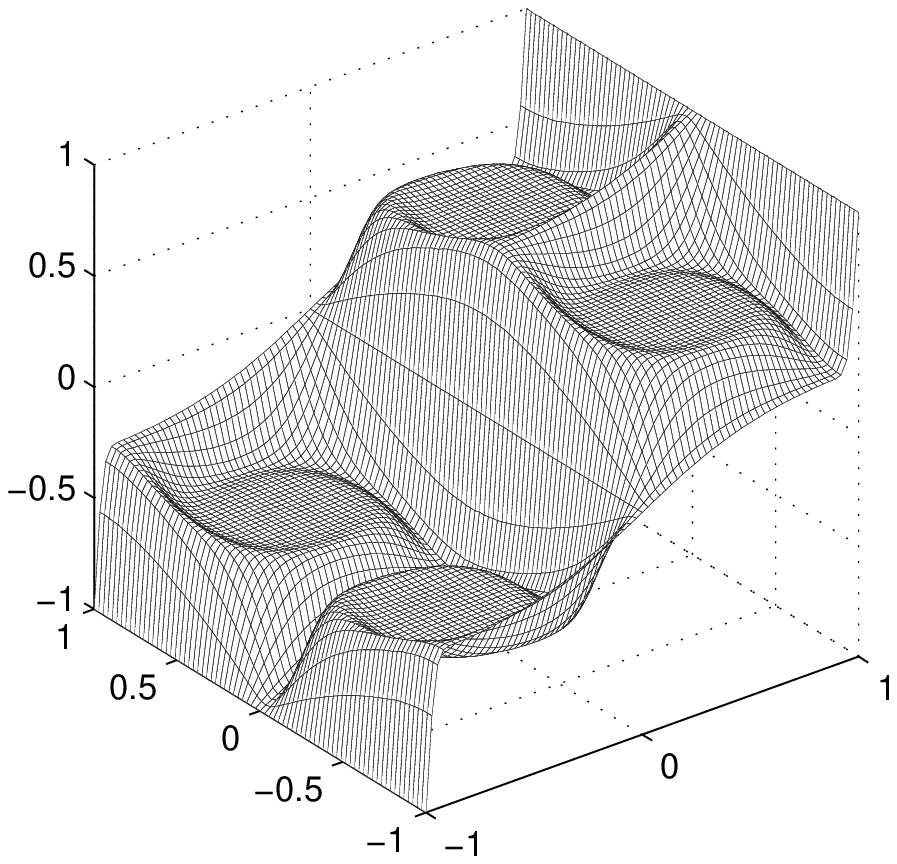}}
\qquad
\subfigure[The cross-section of the plot of the function $y_1+\chi_1(y_1,y_2)-\tfrac{1}{2}\sign y_1$ at $y_1=1/2$.]{\includegraphics[height=7cm]{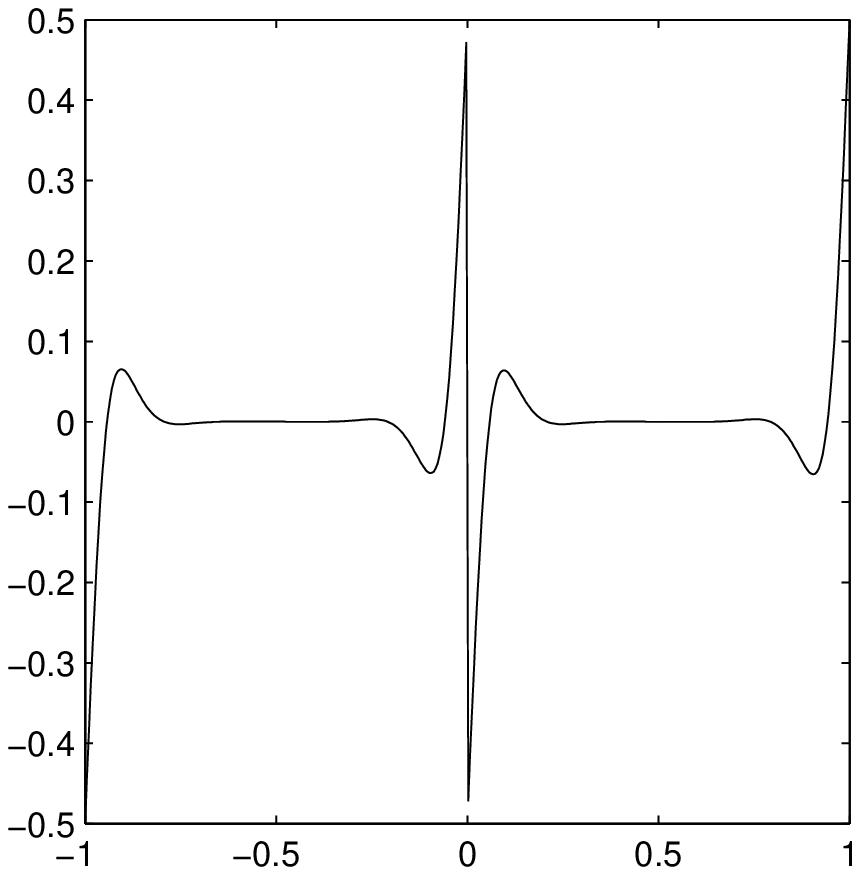}}
\caption{Two plots indicating that $\chi_1 + y_1 - \frac{1}{2} \sign(y_1)$ is small in cell interiors.}\label{fgrxi}
\end{figure}

With this explanation, we proceed to estimate each function individually, 
starting with $\eta$. Since $\divergence v = 0$, Lemma~\ref{lmaDriftIndLinfBound} guarantees
$$
\norm{\eta}_{L^\infty} \leq c \norm{\grad \chi_1}_{L^2}.
$$
We remark that while Lemma~\ref{lmaDriftIndLinfBound} is stated for homogeneous Dirichlet boundary conditions, the proof in~\cite{bblBerestyckiKiselevNovikovRyzhik} goes through verbatim for periodic boundary conditions, provided, of course, we assume our solution is mean-zero. This justifies the application of Lemma~\ref{lmaDriftIndLinfBound} in this context.

Since we know from~\cite{bblFannjiangPapanicolaou} that
$\norm{\grad  \chi_1}_{L^2} = O(A^{1/4})$, 
we immediately obtain
\begin{equation}\label{eqnPhiBoundbis}
\norm{\eta}_{L^\infty} \leq c A^{1/4}.
\end{equation}

Our bound for $\psi_1$ is similar in flavor. Let $\xi_i = \chi_i + y_i - \frac{1}{2} \sign(y_i)$. Then for any $p \geq 1$, we know from~\cite{bblFannjiangPapanicolaou} (see also \cite{bblNovikovPapanicolaouRyzhik}*{Theorem 1.2}) that 
$$
\norm{\xi_i}_{L^p} \leq c A^{-\frac{1}{2p}}.
$$
Since $-A \average{v_i \chi_i} = O(\sqrt{A})$, from Lemma~\ref{lmaDriftIndLinfBound} we have
\begin{equation}\label{eqnVarPsi1bdbis}
\norm{\psi_1}_{L^\infty} \leq c A \left( \norm{\xi_i}_{L^p} + \sum_{i=1}^2 \abs*{\average*{v_i \chi_i}} \right) \leq c A^{1 - \frac{1}{2p}}
\end{equation}
for any $p > 1$.

Finally for $\psi_2$, our aim is to absorb the right hand side into the drift. For $i = 1,2$, Let $f_i = f_{i,A}$ be defined by
$$
f_i(y) =  \frac{ y_i^2 - \abs{y_i} }{2} \qquad\text{if } \frac{1}{\sqrt{A}} \leq \abs{y_i} \leq  1 - \frac{1}{\sqrt{A}},
$$
and extended to be a $C^1$, periodic function on $\R^2$ in the natural way. Set $\theta = \psi_2 + \sum_{i=1}^2 (f_i - \average{f_i})$, then $\theta$ is a periodic, 
mean-zero solution to
$$
-\lap_y \theta + A v \cdot \grad_y \theta = \sum_{i=1}^2 \left( A v_i g_i - \lap_y f_i \right).
$$
where
$$
g_i(y) =   \del_{y_i} f_i - y_i + \frac{1}{2} \sign(y_i) 
$$
Since for any $p \geq 1$, we can explicitly compute
$$
\norm{g_i}_{L^p} \leq c A^{-\frac{1}{2p}} \quad\text{and}\quad \norm{\lap_y f_i}_{L^p} \leq c A^{\frac{1}{2} - \frac{1}{2p}},
$$
by Lemma~\ref{lmaDriftIndLinfBound} we obtain
\begin{equation}\label{eqnVarPsi2bdbis}
\norm{\psi_2}_{L^\infty} \leq 1 + \norm{\theta}_{L^\infty} \leq c A^{1 - \frac{1}{2p}}
\end{equation}
for any $p > 1$.

Thus combining~\eqref{eqnPhiBoundbis}, \eqref{eqnVarPsi1bdbis} and~\eqref{eqnVarPsi2bdbis},  we see $\norm{\tau_{12}}_{L^\infty} \leq c A^{1 - 1/(2p)}$.
Thus using~\eqref{eqnHomogAssumption} and choosing $p = \frac{8 - 2\alpha}{8 - 3\alpha}$ when $0 < \alpha < 8/3$, and $p = \infty$ for $\alpha \geq 8/3$, we see
$$
\frac{\norm{\tau_{12}}_{L^\infty}}{L^2} \leq c L^{-\frac{\alpha}{4}}
$$
proving~\eqref{eqnCorrectionTermBoundsbis}. This completes the proof.
\end{proof}

\subsection*{Exit time from a square.}
Theorem~\ref{ppnExitTimeHomog} follows immediately from Proposition~\ref{ppnExitTimeHomog}.

\begin{proof}[Proof of Theorem~\ref{ppnExitTimeHomog}]
The proof of Theorem~\ref{ppnExitTimeHomog} is now trivial. We simply inscribe a disk $\underline{D}=\{|x|\leq L/2\}$ into the square $D = [-L/2, L/2]^2$, and circumscribe a bigger disk $\overline{D} = \{ \abs{x} \leq L / \sqrt{2} \}$ around $D$. The corresponding exit times satisfy the inequality
$$
\underline{\tau}(x) \leq \tau(x) \leq \overline{\tau}(x), \quad\text{for all $x \in \underline{D}$.}
$$
Using the bounds obtained from Proposition~\ref{prop-disk} applied to $\underline{D}$ and $\overline{D}$, the inequality~\eqref{071104} follows.
\end{proof}

\section{The eigenvalue in the homogenization regime}\label{sec:homeig}
\subsection{The lower bound}\label{sxnHomogLowerBound}
The lower bound for the eigenvalue stated in Theorem~\ref{thmEvalHomog} follows, quickly from the upper bound on the expected exit time.

\begin{proof}[Proof of the lower bound in Theorem~\ref{thmEvalHomog}]
We claim that in general, we have the principal eigenvalue and expected exit time satisfy
\begin{equation}\label{071304}
\lambda \geq  \frac{1}{\|\tau\|_{L^\infty}}.
\end{equation}
To see this, pick any $\epsilon > 0$, and suppose for contradiction that $\lambda < 1 / \norm{\tau + \epsilon}_{L^\infty}$. Then, 
$$
-\Delta (\tau + \epsilon) + A v\cdot\nabla (\tau + \epsilon) = 1\geq\frac{1}{\|\tau + \epsilon \|_{L^\infty}} \left(\tau + \epsilon\right).
$$
Also
$$
- \lap \varphi + A v \cdot \grad \varphi = \lambda \varphi \leq \frac{1}{\norm{\tau + \epsilon}_{L^\infty}} \varphi.
$$
Rescaling $\varphi$ if necessary to ensure $\norm{\varphi}_{L^\infty} \leq \epsilon$, we see have $\varphi \leq \tau + \epsilon$ in $D$. Thus Perron's method implies the existence of a function $\phi$ such that
\begin{IEEEeqnarray*}{c?s}
- \lap \phi + A v \cdot \grad \phi = \frac{1}{\norm{\tau + \epsilon}_{L^\infty}} \phi & in $D$,\\
\phi = 0 & on $\del D$,\\
\varphi \leq \phi \leq \tau & in $D$
\end{IEEEeqnarray*}
This immediately implies $1 / \norm{\tau + \epsilon}_{L^\infty}$ equals the principal eigenvalue $\lambda$, which contradicts our assumption. Thus, for any $\epsilon > 0$, we must have $\lambda \geq 1 / \norm{\tau + \epsilon}_{L^\infty}$. Sending $\epsilon \to 0$, we obtain~\eqref{071304}. Applying Theorem~\ref{ppnExitTimeHomog} concludes the proof.
\end{proof}


\subsection{The upper bound}\label{sxnHomogUpperBound}
In this section we prove the upper bound in Theorem~\ref{thmEvalHomog}. We will do this by using a multi-scale expansion 
of a sub-solution. As we have seen in the preceding sections, our multi-scale expansions are all forced to use a quadratic `slow' profile,
in order to avoid extra terms in the expansion.
This makes the construction of the sub-solution slightly more difficult. As customary with homogenization problems, we rescale the problem so that the cell size goes to $0$, and the domain is fixed.

\begin{lemma}\label{lmaPsi}
Let $h > 0$, and $\psi$ be the solution of
\begin{equation}\label{eqnPsiBound}
\left\{
\begin{aligned}
-\lap \psi + A L v(Lx) \cdot \grad \psi &= \chi_{B_{1-h}} &&\text{in } B_1\\
\psi &= 0 &&\text{on }\del B_1,\\
\end{aligned}\right.
\end{equation}
where $B_r=\{|x|\leq r\}$,  and $\chi_S$ is the characteristic function of the set $S$. 
Assume that $A$ and $L$ vary such that~\eqref{eqnHomogAssumption} holds. Then there exists $h > 0$, and $c = c(h) > 0$ such that
\begin{equation}\label{071106}
\psi(x) \geq \frac{c}{\sqrt{A}} \quad\text{for all } x \in B_{1-h}.
\end{equation}
provided $A$ and $L$ are sufficiently large.
\end{lemma}

Lemma~\ref{lmaPsi} immediately implies the desired upper bound. We present this argument below before delving into the technicalities of Lemma~\ref{lmaPsi}.

\begin{proof}[Proof of the upper bound in Theorem~\ref{thmEvalHomog}]
Let $\varphi_1$, $\mu_1$ be the principal eigenfunction and the principal eigenvalue respectively for the rescaled problem
\begin{equation}\label{eqnEvalBall}
\left\{
\begin{aligned}
-\lap \varphi_1 + A L v(Lx) \cdot \grad \varphi_1 &= \mu_1 \varphi_1 &&\text{in } B_1\\
\varphi_1 &= 0 &&\text{on }\del B_1,\\
\varphi_1 &> 0 && \text{in }B_1.
\end{aligned}\right.
\end{equation}
Assume, for contradiction, $\mu_1 > \frac{\sqrt{A}}{c}$, where $c$ is the constant in Lemma~\ref{lmaPsi}, then
$$
-\lap \varphi_1 + A L v(Lx) \cdot \grad \varphi_1 = \mu_1 \varphi_1 \geq \frac{\sqrt{A}}{c} \varphi_1.
$$
Also, if $\psi$ is the function from Lemma~\ref{lmaPsi}, then by the maximum principle, $\psi > 0$ in $B_{1-h}$. 
Hence,
$$
-\lap \psi + A L v(Lx) \cdot \grad \psi = \chi_{B_{1-h}} \leq \frac{\sqrt{A}}{c} \psi.
$$
By the Hopf lemma, we know $\frac{\del \varphi_1}{\del n} < 0$ on $\del B_1$, and so $\varphi_1$ can be rescaled to ensure $\varphi_1 \geq \psi$. Perron's method now implies that there exists a function $\phi$ that satisfies
\begin{equation}\label{eqnEvalBallbis2}
\left\{
\begin{aligned}
-\lap \phi + A L v(Lx) \cdot \grad \phi&= \frac{\sqrt{A}}{c}\phi &&\text{in } B_1\\
\phi &= 0 &&\text{on }\del B_1,\\
\phi &> 0 && \text{in }B_1,
\end{aligned}\right.
\end{equation}
and, in addition, $\psi(x)\leq\phi(x)\leq\varphi_1(x)$. Therefore, $\mu_1=\sqrt{A}/c$ is the principal eigenvalue, which contradicts our assumption $\mu_1 > \sqrt{A}/c$. Hence $\mu_1 \leq \sqrt{A}/c$.

Now rescaling back so the cell size is $1$, let $\lambda'$ and $\phi'$ be the principal eigenvalue and principal eigenfunction respectively of the problem~\eqref{eqnEval} on the ball of radius $L/2$. Since $\lambda' = 4 \mu_1 / L^2$, we have $\lambda' \leq 4\sqrt{A}/(c L^2)$. Finally, let $D$ be the square with side length $L$, and $\lambda$, $\varphi$ are the principal eigenvalue and eigenfunction respectively of the problem~\eqref{eqnEval} on $D$. Then, since $B_{L/2} \subset D$, the principal eigenvalues must satisfy $\lambda \leq \lambda'$, from which the theorem follows.
\end{proof}
 
It remains to prove Lemma~\ref{lmaPsi}.

\begin{proof}[Proof of Lemma~\ref{lmaPsi}]
Let $\theta = \tau_1 - \psi$, where $\tau_1$ is the solution of~\eqref{eqnTau1}, the expected exit time problem from the unit disk. Now rescaling~\eqref{eqnTau1Bound-bis} to the ball of radius $1$, (or directly using~\eqref{eqnTau1B1}, which was what lead to~\eqref{eqnTau1Bound-bis}), we obtain
\begin{equation}\label{eqnTau1Bound}
\abs*{\tau_1(x) - \frac{1}{2 \trace(\bar\sigma(A))} \left( 1 - \abs{x}^2 \right)} \leq \frac{c_3 L^{-\alpha/4}}{\sqrt{A}}
\end{equation}
provided $A$ and $L$ are large enough and satisfy~\eqref{eqnHomogAssumption}. Here $c_3 > 0$ is a fixed constant independent of $A$ and $L$. Thus, \eqref{071106} will follow if we show that
\begin{equation}\label{eqnThetaBound1}
\norm{\theta}_{L^\infty(B_{1-h})} \leq (1 - \epsilon') \inf_{x\in B_{1-h}} \tau_1(x)
\end{equation}
for some small $\epsilon' > 0$. Observe that~\eqref{eqnTau1Bound} implies that the right hand side of~\eqref{eqnThetaBound1} 
is $O(h/\sqrt{A})$. Therefore, to establish~\eqref{eqnThetaBound1}, it suffices to show that there exists constants 
$h_0 > 0$ and $c > 0$ such that for all $h \leq h_0$, there exists $A_0 = A_0(h)$ and $L_0 = L_0(h)$ such that
\begin{equation}\label{eqnThetaBound2}
\norm{\theta}_{L^\infty(B_{1-h})} \leq \frac{c}{\sqrt{A}} h^{3/2},
\end{equation}
provided $A \geq A_0$, $L \geq L_0$ and~\eqref{eqnHomogAssumption} holds. Above any power of $h$ strictly larger
than $1$ will do; our construction below obtains $h^{3/2}$, however, in reality one would expect the power to be $h^2$.

We will obtain~\eqref{eqnThetaBound2} by considering a Poisson problem on the annulus 
\[
\annulus_{1-2h,1}=\{1-2h\leq |x|\leq 1\}.
\] 
If we impose a large enough constant boundary condition on the inner boundary, the (inward) normal derivative 
will be negative on $\del B_{1-2h}$. Now, if we extend this function inward by a constant, we will have a
super-solution giving the desired estimate for $\theta(x)$. We first state a lemma guaranteeing the sign of the normal 
derivative of an appropriate Poisson problem.
\begin{lemma}\label{lmaTheta1Bound}
There exists $h_0$ and $c_2 > 0$, such that for all $h < h_0$, there exists $A_0$, $L_0 > 0$ such that the solution 
$\theta_1$ of the PDE
\begin{equation}\label{eqnTheta1}
\left\{
\begin{aligned}
-\lap \theta_1 + A L v(Lx) \cdot \grad \theta_1 &= \chi_{\annulus_{1, 1-h}} &&\text{in } \annulus_{1-2h,1}\\
\theta_1 &= 0 &&\text{on }\del B_1,\\
\theta_1 &= \frac{c_2}{\sqrt{A}} h^{3/2} &&\text{on }\del B_{1-2h},\\
\end{aligned}\right.
\end{equation}
satisfies
$$
\frac{\del \theta_1}{\del r} \leq 0 \quad\text{on } \del B_{1-2h},
$$
provided $L \geq L_0$, $A \geq A_0$ and~\eqref{eqnHomogAssumption} holds. 
Here $\frac{\del}{\del r}$ denotes the derivative with respect to the radial direction.
Moreover, the function $\theta_1$ attains its maximum on $|x|=1-2h$, and $\theta_1(x)\leq  {c_2}  h^{3/2}/{\sqrt{A}}$
for all $x\in\annulus_{1-2h,1}$.
\end{lemma}

Now, postponing the proof of Lemma~\ref{lmaTheta1Bound}, we prove~\eqref{eqnThetaBound2}. Choose $h$ small, and $A, L$ large, as guaranteed by Lemma~\ref{lmaTheta1Bound}, and define $\bar\theta$ by
$$
\bar\theta(x) =
\begin{dcases}
\theta_1(x) & \text{when } \abs{x} \geq 1 - 2h\\
\frac{c_2}{\sqrt{A}} h^{3/2}  & \text{when } \abs{x} < 1 - 2h.
\end{dcases}
$$
where $\theta_1$ and $c_2$ are as in Lemma~\ref{lmaTheta1Bound}. Then $\bar\theta \in C(B_1) \cap C^2(B_{1-2h} \cup \annulus_{1-2h,1})$, and
$$
\left(-\lap + A L v(Lx) \cdot \grad\right) \bar\theta (x) =
\begin{dcases}
1 &\text{when } \abs{x} \geq 1 - h\\
0 &\text{when } \abs{x} < 1 - h \;\&\; \abs{x} \neq 1-2h
\end{dcases}
$$
Further, when $\abs{x} = 1- 2h$,
$$
\frac{\del \bar\theta}{\del r^-}  = 0
\quad\text{and}\quad
\frac{\del \bar\theta}{\del r^+}  \leq 0
$$
where the second inequality follows from Lemma~\ref{lmaTheta1Bound}. Thus $\bar\theta$ is a viscosity super solution to the PDE
$$
\left\{
\begin{aligned}
-\lap \theta + A L v(Lx) \cdot \grad \theta &= \chi_{\annulus_{1, 1-h}} &&\text{in } B_1\\
\theta &= 0 &&\text{on }\del B_1.
\end{aligned}\right.
$$
By the comparison principle, we must have $\bar\theta \geq \theta$, which immediately proves~\eqref{eqnThetaBound2}. 
From this~\eqref{eqnThetaBound1} follows, and using~\eqref{eqnTau1Bound} we obtain~\eqref{eqnPsiBound}, 
concluding the proof.
\end{proof}

It remains to prove Lemma~\ref{lmaTheta1Bound}. Roughly speaking, if we choose the constant $c_2$ sufficiently large, the function $\theta_1$ is nearly harmonic. The inhomogeneity of the boundary conditions dominates the right side of the 
equation. A ``nearly harmonic''  function should attain its maximum on the boundary, implying the conclusion of Lemma~\ref{lmaTheta1Bound}.

The reason we believe the constant $c_2 h^{3/2} / \sqrt{A}$ is large enough, is because the homogenized exit time from the annulus is quadratic in the width of the annulus. Unfortunately, the slow profile is not quadratic in Cartesian coordinates, and so the best we can do is obtain upper and lower bounds, which need not be sharp. We begin by showing that the expected exit time from an annulus of width $h$ grows like $h^{3/2}$. While we certainly don't expect the exponent $3/2$ to be sharp, any exponent strictly larger than $1$ will suffice for our needs.
\begin{lemma}\label{lmaAnnExitTime}
Let $\annulus_{1-h, 1}$ be the annulus $\annulus_{1-h,1}\defeq B_1 \setminus
B_{1-h}$, and $\tauann$ be the solution of the Poisson problem
\begin{equation}\label{eqnTauAnn}
\left\{
\begin{aligned}
-\lap \tauann + A L v(Lx) \cdot \grad \tauann &= 1 &&\text{in } \annulus_{1-h,1}\\
\tauann &= 0 &&\text{on }\del \annulus_{1-h, 1}
\end{aligned}\right.
\end{equation}
Suppose $L$ and $A$ vary so that~\eqref{eqnHomogAssumption} holds. Then there exists constants 
$h_0 > 0$ and $c > 0$ such that for all $h < h_0$, there exists $A_0 = A_0(h)$ and $L_0 = L_0(h)$ such that
$$
\norm{\tauann}_{L^\infty(\annulus_{1-h, 1})} \leq \frac{c_1}{\sqrt{A}} h^{3/2}
$$
provided $L \geq L_0$, $A \geq A_0$ and~\eqref{eqnHomogAssumption} holds.
\end{lemma}
\begin{proof}
The main idea behind the proof is that as $A, L \to \infty$, we know that $\tau_1$ tends to an explicit (homogenized)
parabolic profile and is 
constant on $\del B_{1-h}$. Now if we subtract off a harmonic function with these boundary values, then we should get a 
super solution for $\tauann$. Finally, we will show that a harmonic function with constant boundary values grows 
linearly near $\del B_1$, at the same rate as $\tau_1$. Thus the above super solution will give an upper bound for $\tauann$ which is super-linear
in the annulus width.

We proceed to carry out the details. Let $\eta'$ be the solution of
\begin{equation}\label{eqnEta}
\left\{
\begin{aligned}
-\lap \eta' + A L v(Lx) \cdot \grad \eta' &= 0 &&\text{in } \annulus_{1-h,1}\\
\eta' &= 0 &&\text{on }\del B_1\\
\eta' &= \frac{2h - h^2}{2 \trace(\bar\sigma(A))} &&\text{on }\del B_{1-h},
\end{aligned}\right.
\end{equation}
and define
$$
\bartauann = \tau_1 - \eta' + \frac{c_3 L^{-\alpha/4}}{\sqrt{A}},
$$
where $c_3$ is as in~\eqref{eqnTau1Bound}. Then $\bartauann$ satisfies
$$
\left\{
\begin{aligned}
-\lap \bartauann + A L v(Lx) \cdot \grad \bartauann &= 1 &&\text{in } \annulus_{1-h,1}\\
\bartauann &\geq 0 &&\text{on }\del B_1\\
\bartauann &\geq 0 &&\text{on }\del B_{1-h}.
\end{aligned}\right.
$$
The first boundary condition follows because both $\eta'$ and $\tau_1$ are $0$ on $\del B_1$. The second follows 
from~\eqref{eqnTau1Bound} and the boundary condition for $\eta'$. Thus, the maximum principle immediately implies that 
$\bartauann \geq \tauann$. 

Since~\eqref{eqnTau1Bound}  gives the asymptotics for $\tau_1$, to conclude the proof we need a lower bound on $\eta'$ that is `linear' in the radial direction near $\partial B_1$. We separate this estimate as a lemma.
\begin{lemma}\label{lmaEta}
Let $\eta$ be the solution of
\begin{equation}\label{eqnEtaNormalized}
\left\{
\begin{aligned}
-\lap \eta + A L v(Lx) \cdot \grad \eta &= 0 &&\text{in } \annulus_{1-h,1}\\
\eta &= 0 &&\text{on }\del B_1\\
\eta &= h &&\text{on }\del B_{1-h},
\end{aligned}\right.
\end{equation}
Then there exists a constant $c$, independent of $h$, $A$ and $L$, such that
\begin{equation}\label{071202}
\eta(x) \geq 1 - \abs{x} - c \left( h^{3/2} + \frac{L^{-\alpha/4}}{\sqrt{h}} \right)
\end{equation}
when $L$ and $A$ are sufficiently large.
\end{lemma}
Returning to the proof of Lemma~\ref{lmaAnnExitTime}, we see that Lemma~\ref{lmaEta} gives
$$
\eta'(x) \geq \frac{2 - h }{2\trace(\bar\sigma(A))}  \left[ 1 - \abs{x} - c \left( h^{3/2} + \frac{L^{-\alpha/4}}{\sqrt{h}} \right) \right].
$$
Now using the above and~\eqref{eqnEffectiveDiffusivityAsym}, it follows that
\begin{align*}
\tauann(x) &\leq \bartauann(x) = \tau_1(x) - \eta'(x) + \frac{c L^{-\alpha/4}}{\sqrt{A}}\\
    &\leq \frac{1}{2\trace(\bar\sigma(A))} \left( 1 - \abs{x}^2 - (2 - h)(1 - \abs{x}) + c h^{3/2}\right) + \frac{1}{\sqrt{A}}\left( c L^{-\alpha/4} + \frac{c L^{-\alpha/4}}{\sqrt{h}} \right) \\
    &= \frac{1}{2\trace(\bar\sigma(A))} \left( \left( 1 - \abs{x} \right) \left(h - (1 - \abs{x}) \right)+ c h^{3/2} \right) + \frac{1}{\sqrt{A}}\left( \frac{c L^{-\alpha/4}}{\sqrt{h}} \right) \\
    & \leq \frac{1}{2\trace(\bar\sigma(A))} \left( h^2 + c h^{3/2} \right) + \frac{1}{\sqrt{A}}\left( \frac{c L^{-\alpha/4}}{\sqrt{h}} \right)
\end{align*}
obtaining Lemma~\ref{lmaAnnExitTime} as desired.
\end{proof}

To complete the proof of Lemma~\ref{lmaAnnExitTime}, we need to prove Lemma~\ref{lmaEta}. We do this next.

\begin{proof}[Proof of Lemma~\ref{lmaEta}]
We will construct a sub-solution of equation~\eqref{eqnEtaNormalized} in small rectangles overlapping $\annulus_{1-h,1}$. 
For convenience, we now shift the origin to $(-1,0)$, and consider new coordinates $(x_1', x_2') \defeq (x_1 + 1, x_2)$. 
In these coordinates, let $R$ be the rectangle of height $2 h^{3/4}$, width $h$ and top left corner $(\rho_0, h^{3/4})$, 
where $\rho_0 = 1 - (1 - h^{3/2})^{1/2}$ (see Figure~\ref{fgrEta}).

\begin{figure}[htb]
\center{\includegraphics[width=5cm]{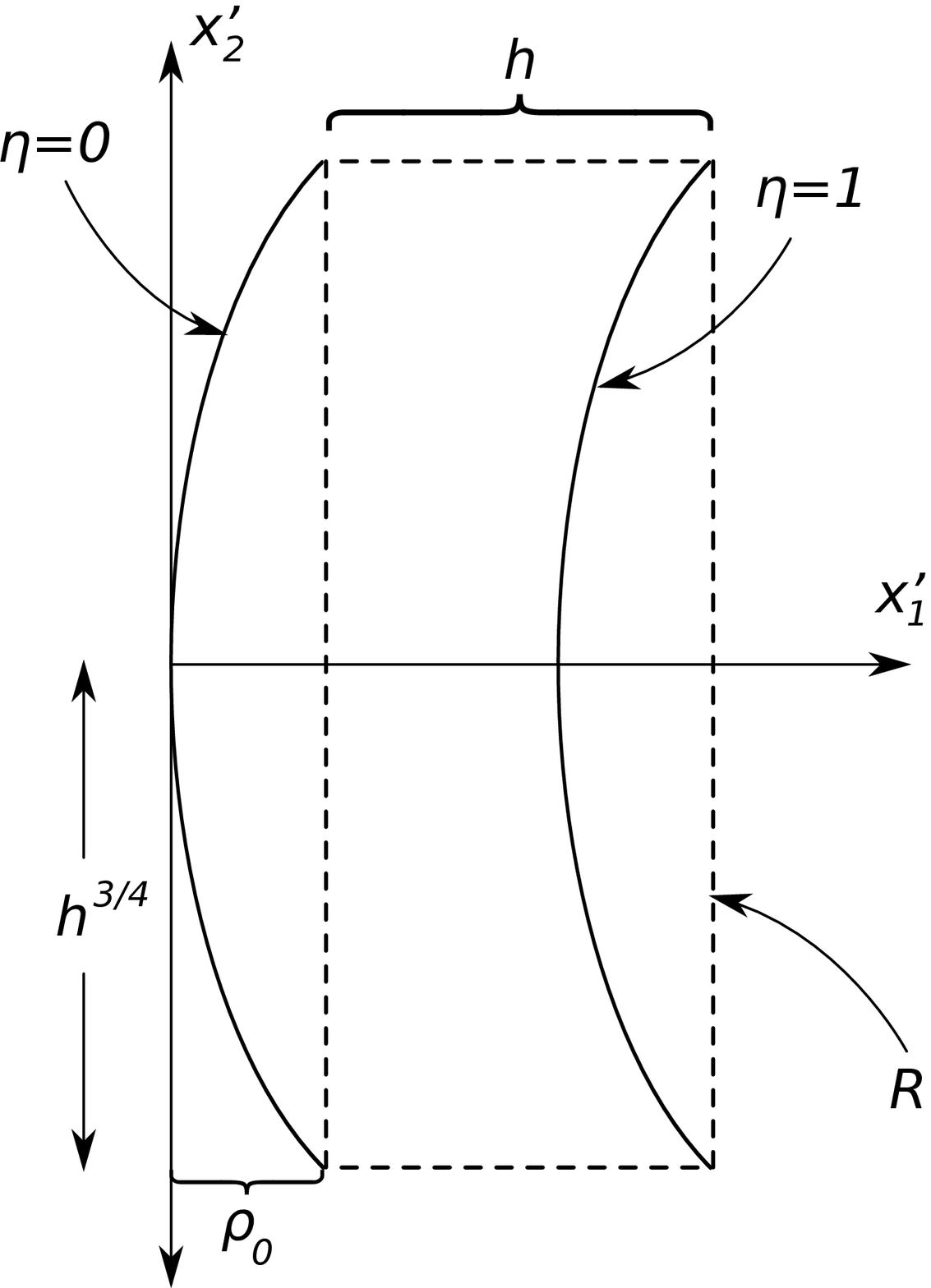}}
\caption{Domain for $\etal$.}\label{fgrEta}
\end{figure}

We will construct a function $\etal$ such that
\begin{equation}\label{eqnEtal}
\left\{
\begin{aligned}
-\lap \etal + A L v(Lx) \cdot \grad \etal &= 0 &&\text{in } \annulus_{1-h,1},\\
\etal &\leq h &&\text{on the right boundary of } R,\\
\etal &\leq 0 &&\text{on the other three boundaries of } R,
\end{aligned}\right.
\end{equation}
and $\etal$ satisfies the linear growth condition
\begin{equation}\label{eqnEtalLowerBound}
\etal(x_1', 0) \geq x_1' - c\left(h^{3/2} + \frac{L^{-\alpha/4}}{\sqrt{h}}\right)\quad\text{when } x_1' \in [ \rho_0, \rho_0 + h],
\end{equation}
for some constant $c$ independent of $L$, $A$ and $h$.

Before proving that the function $\etal$ exists, we remark that by the maximum principle, we $\etal \leq \eta$ on 
$R \cap \annulus_{1-h,h}$. Moreover, as $\rho_0=O(h^{3/2})$, the estimate \eqref{071202} can be extended 
to $x_1'\in[0,\rho_0]$ as well, possibly by increasing the constant $c$.
This proves Lemma~\ref{lmaEta} when $x$ is on the negative $x_1$-axis. Now, if~\eqref{eqnEtalLowerBound} 
is still valid when the coordinate frame is rotated, our proof of Lemma~\ref{lmaEta} will be complete!\medskip

We will first prove that a function $\etal$ satisfying~\eqref{eqnEtal} and~\eqref{eqnEtalLowerBound} exists. 
We will do this by a multi-scale expansion. Let
$$
\etal(x') = \eta_0(x') + \frac{1}{L} \eta_1(x',y) + \frac{1}{L^2}\eta_2(y) - \frac{c_0 L^{-\alpha/4}}{\sqrt{h}},
$$
where $y = Lx'$ is the fast variable, $\eta_0$ is given by
$$
\eta_0(x_1', x_2') = x_1' - c_1 \rho_0 + \frac{( (x_1')^2 - (x_2')^2 )}{\sqrt{h}},
$$
and $c_0$, and $c_1$ are constants, each independent of $L$, $A$ and $h$, to be chosen later. As before, $\eta_1$ is
$$
\eta_1(x',y) = \sum_{i=1}^2 \chi_i(y) \del_{x_i'} \eta_0(x'),
$$
and $\eta_2$ is the mean $0$, periodic solution to
\begin{equation}\label{071208}
-\lap_y \eta_2 + A v(y) \cdot \grad_y \eta_2 = \sum_{i=1}^2 \left[ 2 \del_{y_i} \chi_j - A \left( v_i \chi_j - \average{v_i \chi_j} \right) \right] \del_{x_i'}\del_{x_j'} \eta_0,
\end{equation}
where $\chi_i$ are the solutions to~\eqref{eqnChi1bis}.

Again, the crucial fact here is that since $\eta_0$ is quadratic, the second derivatives are constant and $\eta_2$ 
becomes independent of the slow variable $x'$. Using this, a direct computation shows that
\begin{equation}\label{071206}
-\lap \etal + A L v \cdot \grad \etal
    = -\lap \eta_0  + A \sum_{i,j} \average{v_i \chi_j} \del_{x_i'} \del_{x_j'} \eta_0.
\end{equation}
Note that by symmetry, $\average{v_1 \chi_1} = \average{v_2 \chi_2}$, and $\average{v_1 \chi_2} = \average{v_2 \chi_1} = 0$.
Further, by our choice of $\eta_0$, we have $\lap \eta_0 = 0$. Hence, the previous equation reduces to
$$
-\lap \etal + A L v \cdot \grad \etal = 0,
$$
as required by the first equation in~\eqref{eqnEtal}.

Next, we show that if $c_1$ is appropriately chosen, we can arrange the boundary conditions claimed 
in~\eqref{eqnEtal} for $\eta_0$. Notice that $\rho_0 = O(h^{3/2})$, and on the top and bottom boundary 
we have $x_2' = \pm h^{3/4}$ and $x_1' \in (\rho_0, \rho_0 + h)$. Thus
$$
x_1' - \frac{(x_2')^2}{\sqrt{h}} \leq \rho_0 \quad\text{and}\quad
\frac{(x_1')^2}{\sqrt{h}} \leq O(h^{3/2}).
$$
So choosing $c_1$ large enough, we can ensure $\eta_0 \leq 0$ on the top and bottom of $R$.

On the left of $R$, we have $x_1' = \rho_0$ and $\abs{x_2} \leq h^{3/4}$. So $(x_1')^2 / \sqrt{h} \leq O(h^{5/2}) = o(\rho_0)$, 
and choosing $c_1$ large we can again ensure $\eta_0 \leq 0$ on the left of $R$. Finally, on the right of $R$, we have 
$x_1' = \rho_0 + h$ and $\abs{x_2'} \leq h^{3/4}$, and we immediately see that for $c_1$ large enough, we have $\eta_0 \leq h$ 
on the right of $R$. Thus $\eta_0$ satisfies the boundary conditions in~\eqref{eqnEtal}.

To see that $\etal$ also satisfies the boundary conditions in~\eqref{eqnEtal}, we need to bound the correctors appropriately. 
Exactly as in the proof of Lemma~\ref{lmaSupSolbis}, we obtain 
\begin{equation}\label{071210}
\norm*{\frac{1}{L} \eta_1 + \frac{1}{L^2} \eta_2}_{L^\infty} \leq \frac{c L^{-\alpha / 4}}{\sqrt{h}},
\end{equation}
where $c > 0$ is independent of $L$, $A$ and $h$. We remark that the extra ${1}/{\sqrt{h}}$ factor arises because 
derivatives of $\eta_0$ are of the order ${1}/{\sqrt{h}}$ and they appear as multiplicative factors in the expressions for 
$\eta_1$ and $\eta_2$.

Consequently, if $c_0$ is chosen to be larger than $c$, we have
\begin{equation}\label{eqnEtaCorrectorBounds}
\frac{1}{L} \eta_1 + \frac{1}{L^2} \eta_2 - \frac{ c_0 L^{-\alpha / 4}}{\sqrt{h}} \leq 0.
\end{equation}
Since $\eta_0$ already satisfies the boundary conditions in~\eqref{eqnEtal}, this immediately implies that $\etal$ must also satisfy 
these boundary conditions. Finally, since $\eta_0$ certainly satisfies~\eqref{eqnEtalLowerBound}, it follows 
from~\eqref{eqnEtaCorrectorBounds}  that $\etal$ also satisfies~\eqref{eqnEtalLowerBound}.
This proves the existence of $\etal$.

Now, as remarked earlier, the only thing remaining to complete the proof of the Lemma is to verify that if the rectangle $R$, 
and the coordinate frame are both rotated arbitrarily about the center of the annulus, then there still exists a function $\etal$ 
satisfying~\eqref{eqnEtal} and~\eqref{eqnEtalLowerBound} in new coordinates. This, however, is immediate. The new 
coordinates can be expressed in terms of the old coordinates as a linear function. Consequently, our initial profile for $\eta_0$ 
will still be a \textit{quadratic} function of the new coordinates. Of course, by the rotational invariance of the Laplacian, 
it will also be harmonic, and the remainder of the proof goes through nearly verbatim. The only modification is that after
the rotation the mixed derivative $\partial_{x_1'}\partial_{x_2'}\eta_0$ no longer vanishes, and the terms involving 
$v_1\chi_2$ and $v_2\chi_1$ do appear in \eqref{071208} and \eqref{071206}. However, they can be treated in an identical fashion,
as in the proof of Lemma~~\ref{lmaSupSolbis} using the precise asymptotics for $\chi_{1}$  and $\chi_2$
from~\cite{bblNovikovPapanicolaouRyzhik}. This concludes the proof.
\end{proof}

Finally, we are ready for the proof of Lemma~\ref{lmaTheta1Bound}.

\begin{proof}[Proof of Lemma~\ref{lmaTheta1Bound}]
For a given $c_2$, let $\eta'$ be the solution of
$$
\left\{
\begin{aligned}
-\lap \eta' + A L v(Lx) \cdot \grad \eta' &= 0 &&\text{in } \annulus_{1-2h,1}\\
\eta' &= 0 &&\text{on }\del B_1\\
\eta' &= \frac{c_2}{\sqrt{A}} h^{3/2} &&\text{on }\del B_{1-2h},
\end{aligned}\right.
$$
Then, we have
\[
\theta_1 - \eta' = 0 \hbox{ on $\del \annulus_{1-2h,1}$},
\]
and 
\[
-\lap\theta_1 + A L v(Lx) \cdot \grad \theta_1 \leq 1 \hbox{ in $\annulus_{1-2h,1}$.}
\]
Consequently $\theta_1 - \eta' \leq \tauann$, where $\tauann$ is the solution of~\eqref{eqnTauAnn} on the annulus $\annulus_{1-2h, 1}$. Thus applying Lemma~\ref{lmaAnnExitTime}, we see
\begin{equation}\label{eqnTheta1EtaPrimeBound}
\theta_1(x) \leq \frac{c_1}{\sqrt{A}} (2h)^{3/2} + \eta'(x)
\end{equation}

The function $\eta'$ decreases at most linearly with $\abs{x}$. This can be seen immediately from an asymptotic expansion for a super solution. Indeed, starting with
$$
\eta_0(x) = \frac{c_2\sqrt{h}}{2\sqrt{A}} (1 - x_1)
$$
and choosing $\eta_1$ and $\eta_2$ as in the proof of Lemma~\ref{lmaEta}, we immediately see that
$$
\eta'(x_1, 0) \leq \frac{c_2\sqrt{h}}{2\sqrt{A}} (1 - x_1) + c \sqrt{h} A^{-1/2} L^{-\alpha/4}.
$$
We remark again that the extra $\sqrt{h}A^{-1/2}$ factor arises from the gradient of $\eta_0$. Now by rotating the 
initial profile $\eta_0$ appropriately, we obtain the linear decrease
\begin{equation}\label{eqnEtaPrimeLinearDecrease}
\eta'(x) \leq \frac{c_2\sqrt{h}}{2\sqrt{A}} (1 - \abs{x}) + c \sqrt{h} A^{-1/2} L^{-\alpha/4}.
\end{equation}
as claimed.

We claim that~\eqref{eqnTheta1EtaPrimeBound} and~\eqref{eqnEtaPrimeLinearDecrease} quickly conclude the proof. 
To see this, note first that equation~\eqref{eqnTheta1EtaPrimeBound} and~\eqref{eqnEtaPrimeLinearDecrease} immediately give
\begin{equation}\label{eqnTheta1Bound2}
\theta_1(x) \leq \frac{c_1}{\sqrt{A}} (2h)^{3/2} + \frac{c_2}{2 \sqrt{A}} h^{3/2} + c \sqrt{h} A^{-1/2} L^{-\alpha/4}
\quad\text{whenever } x\in \annulus_{1-h,1}.
\end{equation}
However, since 
\[
-\lap\theta_1 + A L v(Lx) \cdot \grad\theta_1 = 0
\]
 on $\annulus_{1-2h,1-h}$, the maximum principle implies that $\theta_1$ can not attain it's maximum in the interior of the 
annulus $\annulus_{1-2h,1-h}$. Consequently~\eqref{eqnTheta1Bound2} must hold on the interior of the \textit{entire} 
annulus $\annulus_{1-2h,1}$.

Now if we choose $c_2$ large enough so that $2^{3/2} c_1 < \frac{c_2}{4}$, and then choose $L, A$ large enough so that 
$c A^{-1/2} L^{-\alpha/4} < \frac{c_2}{4}$, we see that $\theta_1$ is forced to attain it's maximum on the inner boundary 
$\del B_{1-2h}$, and the Lemma follows immediately.
\end{proof}

\end{document}